\documentclass[11pt]{amsart}

\usepackage{amssymb, amscd, txfonts}
\usepackage{graphicx}

\numberwithin{equation}{section}

\newtheorem{theorem}{Theorem}[section]
\newtheorem{proposition}[theorem]{Proposition}
\newtheorem{lemma}[theorem]{Lemma}
\newtheorem{corollary}[theorem]{Corollary}

\theoremstyle{definition}
\newtheorem{definition}[theorem]{Definition}

\theoremstyle{remark}
\newtheorem{remark}[theorem]{Remark}
\newtheorem{example}[theorem]{Example}

\newcommand{\Z}{\mathbb{Z}}
\newcommand{\Q}{\mathbb{Q}}

\newcommand{\C}{\mathbb{C}}
\newcommand{\proj}{{\mathbb P}}

\newcommand{\moduli}{\mathcal{M}_{(r, a, \delta )}}
\newcommand{\cover}{\widetilde{\mathcal{M}}_{(r, a, \delta )}}
\newcommand{\SL}{{\rm SL}_{3}}
\newcommand{\PGL}{{\rm PGL}_{3}}
\newcommand{\HP}{\mathcal{O}_{{\mathbb P}^{2}}(1)}

\newcommand{\cubic}{\mathcal{O}_{{\mathbb P}^{2}}(3)}
\newcommand{\sextic}{\mathcal{O}_{{\mathbb P}^{2}}(6)}
\newcommand{\cohomology}{H^{2}(X, \mathbb{Z})}

\begin{document}

\title[]{The unirationality of the moduli spaces of 2-elementary $K3$ surfaces}
\author[]{Shouhei Ma}
\address[S. Ma]{Graduate School of Mathematical Sciences, the University of Tokyo, Tokyo 153-8914, Japan}
\email{sma@ms.u-tokyo.ac.jp}
\address[K.-I. Yoshikawa]{Department of Mathematics, Kyoto University, Kyoto 606-8502, Japan \linebreak 
Korea Institute for Advanced Study, Hoegiro 87, Dongdaemun-gu, Seoul 130-722, Korea}
\email{yosikawa@math.kyoto-u.ac.jp}
\maketitle 
\centerline{with an Appendix by Ken-Ichi Yoshikawa}

\begin{abstract}
We prove that the moduli spaces of $K3$ surfaces with non-symplectic involutions are unirational. 
As a by-product we describe configuration spaces of $5\leq d\leq 8$ points in ${\proj}^2$ as 
arithmetic quotients of type IV. 
\end{abstract}

\maketitle


\section{Introduction}

$K3$ surfaces with non-symplectic involutions were classified by Nikulin \cite{Ni2}, 
and Yoshikawa \cite{Yo1} showed that their moduli spaces are 
Zariski open sets of certain modular varieties of orthogonal type. 
In this paper we prove that those moduli spaces are unirational. 
This work was inspired by a recent result of Yoshikawa on the Kodaira dimensions of those spaces, 
which is presented by him in the Appendix \ref{appendix} of this paper. 
Let us begin by recalling basic definitions. 

Let $X$ be a complex $K3$ surface with an involution $\iota$. 
When $\iota$ acts nontrivially on $H^0(K_X)$, 
$\iota$ is called \textit{non-symplectic}, %
and the pair $(X, \iota)$ is called a \textit{2-elementary $K3$ surface}. 
In this case, the lattice $L_+={\cohomology}^{\iota}$ of $\iota$-invariant cycles is a hyperbolic lattice 
with 2-elementary discriminant form $D_{L_+}$. 
The \textit{main invariant} of $(X, \iota)$ is the triplet $(r, a, \delta)$ where  
$r$ is the rank of $L_+$, 
$a$ is the length of $D_{L_+}$, i.e., $D_{L_+}\simeq({\Z}/2{\Z})^a$, %
and $\delta$ is the parity of $D_{L_+}$. 
Nikulin \cite{Ni2} proved that the deformation type of $(X, \iota)$ is determined by the main invariant $(r, a, \delta)$, 
and he enumerated all main invariants of 2-elementary $K3$ surfaces, 
which are seventy-five in number. %

By the theory of period mapping, 
2-elementary $K3$ surfaces of a fixed main invariant $(r, a, \delta)$ are parametrized by 
the Hermitian symmetric domain associated to a certain lattice $L_-$ of signature $(2, 20-r)$. 
Yoshikawa \cite{Yo1}, \cite{Yo3} determined the correct monodromy group as the orthogonal group ${\rm O}(L_-)$ of $L_-$. 
Consequently, he constructed the moduli space ${\moduli}$ of those pairs $(X, \iota)$
as a Zariski open set of the modular variety defined by ${\rm O}(L_-)$.

The principal result of the present paper is the following. 

\begin{theorem}\label{main}
For every main invariant $(r, a, \delta)$ the moduli space ${\moduli}$ of 
2-elementary $K3$ surfaces of type $(r, a, \delta)$ is unirational. 
\end{theorem}

We recall that  
the 2-elementary $K3$ surfaces in $\mathcal{M}_{(1,1,1)}$ are 
double planes ramified over smooth sextics 
so that $\mathcal{M}_{(1,1,1)}$ is birational to the orbit space 
$|\mathcal{O}_{{\proj}^{2}}(6)|/{\PGL}$, which is unirational. 
This fact is a prototype of Theorem \ref{main}. 
Kond\=o \cite{Ko1} proved the rationality of $\mathcal{M}_{(10,2,0)}$ and $\mathcal{M}_{(10,10,0)}$, 
the latter being isomorphic to the moduli of Enriques surfaces. 
Shepherd-Barron \cite{S-B} practically established the rationality of $\mathcal{M}_{(5,5,1)}$ 
in the course of proving that of the moduli of genus $6$ curves. 
Matsumoto-Sasaki-Yoshida \cite{M-S-Y} constructed general members of $\mathcal{M}_{(16,6,1)}$ 
starting from six lines on ${\proj}^{2}$. 
A similar idea was used by Koike-Shiga-Takayama-Tsutsui \cite{K-S-T-T} to obtain 
general members of $\mathcal{M}_{(14,8,1)}$ 
from four bidegree $(1, 1)$ curves on ${\proj}^1 \times {\proj}^1$. 
In particular, $\mathcal{M}_{(16,6,1)}$ and $\mathcal{M}_{(14,8,1)}$ are also unirational. 

Yoshikawa studied the birational type of ${\moduli}$ in a systematic way %
by using a criterion of Gritsenko \cite{Gr} and Borcherds products. 
He found that ${\moduli}$ has Kodaira dimension $-\infty$ when $13\leq r\leq 17$ and when $r+a=22$, $r\leq 17$. 
After that he suggested to the author to study the birational type of ${\moduli}$ through a geometric approach. %
The present work grew out of this suggestion. 
After Theorem \ref{main} was proved, Yoshikawa and the author decided to write both approaches in this paper. 
Yoshikawa's work is presented in the Appendix \ref{appendix}.  
Now the Kodaira dimensions of some of ${\moduli}$ may be calculated by two methods: 
by modular forms on the moduli spaces, and by the geometry of 2-elementary $K3$ surfaces.

We will prove Theorem \ref{main} by using certain Galois covers of ${\moduli}$  
and isogenies between them.  
The strategy is as follows. 
\begin{enumerate}
  \item Let ${\cover}$ be the modular variety associated to the group $\widetilde{{\rm O}}(L_-)$ of isometries of $L_-$
           which act trivially on the discriminant form. 
           The variety ${\cover}$ is a Galois cover of ${\moduli}$. 
  \item Construct an isogeny ${\cover} \to {\widetilde{\mathcal{M}}_{(r,a',\delta')}}$ 
          when $a'<a$, $\delta=1$, and when $a'<a$, $\delta=\delta'$.   
  \item For each fixed $r$, choose a large $a$ and find a moduli interpretation of (an open set of) ${\cover}$. 
           Then prove that ${\cover}$ is unirational using that interpretation. 
           By step $(2)$ follows the unirationality of ${\widetilde{\mathcal{M}}_{(r,a',\delta')}}$ for $a'<a$. 
  \item The remaining moduli spaces $\mathcal{M}_{(r,a'',\delta'')}$ with $a''>a$, if any, are also proved to be unirational in some way.
\end{enumerate}

One of the advantages of studying the covers ${\cover}$ %
is that we have isogenies between them so that the problem is reduced to fewer modular varieties. 
Those isogenies admit geometric interpretation in terms of twisted Fourier-Mukai partners.  
By this strategy we will derive the unirationality of seventy ${\cover}$ by studying just twenty-two ${\cover}$. 
The remaining five moduli spaces ${\moduli}$, for which we do not know whether the covers ${\cover}$ are unirational, 
are treated in step $(4)$ or already settled (\cite{Ko1}). 
In step $(3)$, we often identify ${\cover}$ with the moduli of 
certain plane sextics endowed with a labeling of the singularities. 
We can attach such geometric interpretations to ${\cover}$ in a fairly uniform manner: 
this is another virtue of studying ${\cover}$. 
We shall explain a general idea of such interpretations (Section \ref{ssec:discri cover}), 
discuss few cases in detail as models (Sections \ref{sec: r <10} and \ref{sec: r=10}), 
and for other cases omit some detail. 

Let us comment on other possible approaches for Theorem \ref{main}. 
Firstly, as explained by Alexeev-Nikulin \cite{A-N}, 
2-elementary $K3$ surfaces with $r+a\leq 20$ are related to log del Pezzo surfaces of index $\leq 2$.  
Thus one might study ${\moduli}$ via the moduli of such surfaces, %
using the explicit description of log del Pezzo surfaces of index $2$ given by Nakayama \cite{Na}. %
Secondly, by using singular curves on ${\proj}^2$ and $\mathbb{F}_n$ as branches (as in this paper), 
for most $(r, a, \delta)$ we can actually find a unirational parameter space that dominates ${\moduli}$. 

In \cite{Ma}, those will be developed further to derive the rationality of 
sixty-seven ${\moduli}$. %
Hence one may establish Theorem \ref{main} also by just studying the remaining moduli spaces. 
However, the proof of rationality is delicate and ad hoc, 
so that the whole proof of unirationality would be lengthy if we do so. 
We here prefer the proof using ${\cover}$ because it is more systematic, short, and self-contained. 

We will relate the covers ${\cover}$ with $r+a=22$ and $r\geq12$ to configuration spaces of points in ${\proj}^2$. 
As a by-product we describe those spaces as arithmetic quotients of type IV.  
To be more precise, 
let $U_{d} \subset ({\proj}^{2})^{d}$ (resp. $V_{d} \subset ({\proj}^{2})^{d}$) be the variety of $d$ ordered points 
of which no three are collinear (resp. only the first three are collinear). 
Let $U_d/G$ and $V_d/G$ denote the quotient varieties for the diagonal actions of $G={\PGL}$. 
Let $L_n$ be the lattice $\langle2\rangle^2\oplus\langle-2\rangle^n$. 

\begin{theorem}\label{by-product}
Let $5\leq d\leq 8$. 
For each $1\leq n\leq8$ there exists an arithmetic group $\Gamma_n\subset{\rm O}(L_n)$   
such that one has birational period maps 
\begin{equation*}
U_d/G \dashrightarrow \mathcal{F}(\Gamma_{2d-8}),  \qquad  
V_d/G \dashrightarrow \mathcal{F}(\Gamma_{2d-9}),  
\end{equation*}
where $\mathcal{F}(\Gamma_n)$ is the modular variety associated to $\Gamma_n$. 
One has $\Gamma_n=\widetilde{{\rm O}}(L_n)$ for $1\leq n \leq 6$,  
and for $n=7, 8$ one has $\Gamma_n\supset\widetilde{{\rm O}}(L_n)$ with 
$\Gamma_n/\widetilde{{\rm O}}(L_n)\simeq\frak{S}_{n-5}$ 
where $\frak{S}_N$ is the symmetric group on $N$ letters. 
\end{theorem}

When $d\leq6$, we recover some results of Matsumoto-Sasaki-Yoshida \cite{M-S-Y}. 
They constructed a period map for $U_6$, 
and then obtained lower-dimensional period maps by degeneration. %
The novel part of Theorem \ref{by-product} is the construction of the period maps for $d=7, 8$ points. 
Also our period maps for $d\leq6$ are derived from the ones for $d=7, 8$, 
and are not identical to the ones of \cite{M-S-Y}.  %
It is a future task to study the whole boundary behavior of the period maps. %

Kond\=o, Dolgachev, and van Geemen \cite{Ko2}, \cite{D-G-K}, \cite{Ko3} 
described the spaces $U_d/G$ for $5\leq d\leq7$ as ball quotients.  
It is also known \cite{D-O} that $U_7/G$ can be described as a Siegel modular variety. 
Thus those spaces $U_d/G$ 
admit (birationally) the structure of an arithmetic quotient in more than one way: 
after suitable compactifications, they may provide examples of ``Janus-like" varieties (cf. \cite{H-W}).  %
In view of the relation with the moduli of del Pezzo surfaces, 
it would also be interesting to study the Weyl group action on $\mathcal{F}(\Gamma_{2d-8})$ 
induced by the period map.

The rest of the paper is structured as follows. 
In Section \ref{sec: preliminary} 
we review the necessary facts concerning lattices, modular varieties, and invariant theory. 
In Section \ref{sec: 2-elemen K3} 
we gather basic results on 2-elementary $K3$ surfaces 
with particular attention to the relation with singular sextic curves. 
The proof of Theorem \ref{main} will be developed from Section \ref{sec: r <10} to Section \ref{sec: r>13}.  
Theorem \ref{by-product} will be proved in Sections \ref{sec: r=12}, \ref{sec: r=13}, and \ref{sec: r>13}.  
In Section \ref{BV 3-fold} 
we deduce the unirationality of the moduli spaces of Borcea-Voisin threefolds 
as a consequence of Theorem \ref{main}. 
In the Appendix \ref{appendix} written by Yoshikawa, the approach by modular forms is presented. 

Otherwise stated, we work in the category of algebraic varieties over ${\C}$.


\section{Preliminaries}\label{sec: preliminary}


\subsection{Lattices}\label{subsec: lattice}

Let $L$ be a \textit{lattice}, i.e., a free ${\Z}$-module of finite rank 
endowed with a non-degenerate integral symmetric bilinear form $( , )$. 
The orthogonal group of $L$ is denoted by ${\rm O}(L)$. 
For an integer $n\ne0$, $L(n)$ denotes the scaled lattice $(L, n( , ))$.  
The lattice $L$ is \textit{even} if $(l, l)\in 2{\Z}$ for all $l \in L$, and \textit{odd} otherwise. 
The dual lattice $L^{\vee}={\rm Hom\/}(L, {\Z})$ of $L$ is canonically embedded in $L\otimes{\Q}$ and contains $L$. 
On the finite abelian group $D_L = L^{\vee}/L$ 
we have the ${\Q}/{\Z}$-valued bilinear form $b_L$ defined by $b_L(x+L, y+L)=(x, y)+{\Z}$. 
We denote by $\widetilde{{\rm O}}(L) \subset {\rm O}(L)$ 
the group of isometries of $L$ which act trivially on $D_L$. 
When $L$ is even, $b_L$ is induced by the quadratic form 
$q_L \colon D_L \to {\Q}/2{\Z}$, $q_L(x+L) = (x, x)+2{\Z}$, 
which is called the {\it discriminant form\/} of $L$. 
We denote by $r_L\colon{\rm O}(L)\to{\rm O}(D_L, q_L)$ the natural homomorphism. 

\begin{proposition}[\cite{Ni1}]\label{glue}
Let $\Lambda$ be an even unimodular lattice and $L$ be a primitive sublattice of $\Lambda$ 
with the orthogonal complement $M$. 
Then one has a natural isometry $\lambda\colon(D_{L}, q_{L}) \simeq (D_{M}, -q_{M})$ defined by the relation 
$x+\lambda(x) \in \Lambda, x \in D_{L}$.  
For two isometries $\gamma_L\in{\rm O}(L)$ and $\gamma_M\in{\rm O}(M)$, 
the isometry $\gamma _{L} \oplus \gamma _{M}$ of $L \oplus M$ 
extends to that of $\Lambda$ 
if and only if 
$r_L(\gamma_L) = \lambda^{-1}\circ r_M(\gamma_M)\circ\lambda$. 
\end{proposition}

A lattice $L$ is called \textit{2-elementary} if $D_L$ is 2-elementary, 
i.e., $D_L \simeq ({\Z}/2{\Z})^a$ for some $a\geq0$. 
The \textit{main invariant} of an even 2-elementary lattice $L$ is 
the quadruplet $(r_+, r_-, a, \delta)$ where 
$(r_+, r_-)$ is the signature of $L$, 
$a$ is the length of $D_L$ as above, 
and $\delta$ is defined by 
$\delta=0$ if $q_L(D_L) \subset {\Z}/2{\Z}$ and $\delta=1$ otherwise. 
By Nikulin \cite{Ni1}, the isometry class of $L$ is uniquely determined by the main invariant 
if $L$ is indefinite. 
When $L$ is hyperbolic, 
we also call the triplet $(1+r_{-}, a, \delta)$ the main invariant of $L$. 
In this paper we often use the following 2-elementary lattices with basis: 
\begin{eqnarray}
 M_n & = &\langle2\rangle\oplus\langle-2\rangle^{n-1} = \langle h, e_1, \cdots, e_{n-1}\rangle,  \label{Cremona lattice} \\
 U(2) & = &\langle u, v\rangle,  \label{U(2)} 
\end{eqnarray}
where $\{ h, e_1, \cdots, e_{n-1}\}$ are orthogonal basis with $(h, h)=2$ and $(e_i, e_i)=-2$, 
and $\{ u, v\}$ are basis with $(u, u)=(v, v)=0$ and $(u, v)=2$. 
Let
\begin{equation}\label{K3 lattice}
\Lambda_{K3} = U^{3} \oplus E_8^2 
\end{equation}
be the even unimodular lattice of signature $(3, 19)$ 
where $U$ is the hyperbolic plane (the scaling of $U(2)$ by $\frac{1}{2}$)  
and $E_8$ is the rank $8$ even negative-definitive unimodular lattice.  
The following assertion is due to Nikulin.

\begin{proposition}[\cite{Ni1}, \cite{Ni2}]\label{unique emb}
Let $L$ be an even hyperbolic 2-elementary lattice. 
If a primitive embedding $L \hookrightarrow \Lambda_{K3}$ exists, 
then it is unique up to the action of ${\rm O}(\Lambda_{K3})$. 
\end{proposition}

\subsection{Orthogonal modular varieties}\label{subsec: modular variety}

Let $L$ be a lattice of signature $(2, r_{-})$ 
and let $\Gamma \subset {\rm O}(L)$ be a finite-index subgroup. 
The group $\Gamma$ acts properly discontinuously on the complex manifold 
\begin{equation*}\label{period domain}
\Omega _{L} =  
   \{  \ {\C}\omega \in {\proj}(L\otimes{{\C}}) \ \vert \ 
           (\omega , \omega )=0,  \ (\omega , \bar{\omega})>0 \ \} . 
\end{equation*}
The domain $\Omega _{L}$ has two connected components, say $\Omega_{L}^{+}$ and $\Omega_{L}^{-}$.  
We denote by $\Gamma^{+}$ the group of those isometries in $\Gamma$ which preserve $\Omega_{L}^{+}$. 
The quotient space 
\begin{equation}\label{def: modular variety}
\mathcal{F}_{L}(\Gamma^{+}) = \Gamma^{+} \backslash \Omega_{L}^{+} 
\end{equation}
is a normal quasi-projective variety of dimension $r_-$, by \cite{B-B}, 
called the modular variety associated to $\Gamma^{+}$. 
When the lattice $L$ is understood from the context, 
we abbreviate $\mathcal{F}_{L}(\Gamma^{+})$ as  $\mathcal{F}(\Gamma^{+})$. 

\begin{proposition}\label{construct isogeny}
Let $L$ be a finite-index sublattice of a lattice $M$ of signature $(2, r_-)$. 
Then there exists a finite surjective morphism 
$\mathcal{F}(\widetilde{{\rm O}}(L)^+) \to \mathcal{F}(\widetilde{{\rm O}}(M)^+)$. 
\end{proposition}

\begin{proof}
We have the sequence $L \subset M \subset M^{\vee} \subset L^{\vee}$ of inclusions. 
If we regard the finite groups $G_{1}=M/L$ and $G_{2}=M^{\vee}/L$ as subgroups of $D_{L}$, 
then we have $G_{2} = \{ x \in D_{L}, b_{L}(x, G_{1})\equiv 0 \}$ 
and the induced bilinear form $(G_{2}/G_{1}, b_{L})$ is canonically isometric to $(D_{M}, b_{M})$. 
Since the isometries in $\widetilde{{\rm O}}(L)$ act trivially on both $G_{1}$ and $G_{2}$, 
they preserve the overlattice $M$ of $L$,  
and as isometries of $M$ act trivially on $D_{M}$. 
Thus we have a finite-index embedding 
$\widetilde{{\rm O}}(L) \hookrightarrow \widetilde{{\rm O}}(M)$ of groups.  
Via the natural identification $\Omega_{L}=\Omega_{M} \subset {\proj}(L\otimes{{\C}}) = {\proj}(M\otimes{{\C}})$, 
this embedding induces a finite morphism 
$\mathcal{F}(\widetilde{{\rm O}}(L)^+) \to \mathcal{F}(\widetilde{{\rm O}}(M)^+)$. 
\end{proof}

The following proposition was used by Kond\=o \cite{Ko1} to prove the rationality of 
the moduli space of Enriques surfaces. 

\begin{proposition}\label{duality of 2-ele lattice}
Let $L$ be an even 2-elementary lattice of signature $(2, r_{-})$. 
Then the lattice $M=L^{\vee}(2)$ is 2-elementary and we have 
$\mathcal{F}({\rm O}(L)^+) \simeq \mathcal{F}({\rm O}(M)^+)$. 
\end{proposition}

\begin{proof}
Since $L(2)\subset M \subset \frac{1}{2}L(2)=M^{\vee}$, 
we see that $M$ is 2-elementary. 
We have the coincidence ${\rm O}(L)={\rm O}(L^{\vee})$ in ${\rm O}(L\otimes{\Q})$ 
because of the double dual relation $L^{\vee\vee}=L$. 
Thus we have 
$\mathcal{F}_L({\rm O}(L)^+)\simeq\mathcal{F}_{L^{\vee}}({\rm O}(L^{\vee})^+)\simeq\mathcal{F}_M({\rm O}(M)^+)$.
\end{proof}


\subsection{Geometric Invariant Theory}\label{subsec: GIT}

We review some facts from Geometric Invariant Theory. 
Throughout this section 
let $X$ be a variety acted on by a reductive algebraic group $G$. 
A $G$-invariant morphism $\pi\colon X \to Y$ to a variety $Y$ is a \textit{geometric quotient} of $X$ by $G$ if 
(i) $\pi $ is surjective, 
(ii) $\mathcal{O}_{Y} \simeq (\pi _{\ast}\mathcal{O}_{X})^{G}$, 
(iii) a subset $U \subset Y$ is open if $\pi ^{-1}(U) \subset X$ is open, 
and (iv) the fibers of $\pi $ are the $G$-orbits.
We sometimes denote $Y = X/G$ and omit $\pi$.  
A geometric quotient $\pi\colon X \to Y$ enjoys the following universality: 
for every $G$-invariant morphism $f\colon X \to Z$ 
there exists a unique morphism $g\colon Y \to Z$ with $g \circ \pi = f$. 
In particular, a geometric quotient, if it exists, is unique up to isomorphism.  

Let $L$ be an ample $G$-linearized line bundle on $X$.  
A point $x \in X$ is \textit{stable} (with respect to $L$) if 
(i) the stabilizer $G_x$ is a finite group, and 
(ii) there is a $G$-invariant section $s \in H^{0}(L^{\otimes n})^{G}$ for some $n>0$ such that 
$s(x) \ne 0$ and that the action of $G$ on $X_{s} = \{ x' \in X, s(x') \ne 0 \}$ is closed. 
The open set of stable points is denoted by $X^{s}(L)$.

\begin{theorem}[\cite{GIT}]\label{thm: GIT}
Let $X, G, L$ be as above. 
Then a geometric quotient $X^{s}(L)/G$ of $X^{s}(L)$ exists and is a quasi-projective variety. 
\end{theorem}

\begin{lemma}\label{stab & fin mor}
Let $f\colon X \to Y$ be a $G$-equivariant finite morphism of $G$-varieties. 
Suppose we have an ample $G$-linearized line bundle $L$ on $Y$ such that $Y = Y^{s}(L)$. 
Then we have $X = X^{s}(f^{\ast}L)$. 
In particular, we have a geometric quotient $X/G$. 
\end{lemma}

\begin{proof}
Note that $f^{\ast}L$ is ample and naturally $G$-linearized. 
For every $x \in X$ the stabilizer $G_{x}$ is a subgroup of $G_{f(x)}$ and hence is finite. 
For an invariant section $s \in H^{0}(Y, L^{\otimes n})^{G}$ 
with $s(f(x)) \ne 0$ and with closed $G$-action on $Y_s$, 
we have $f^{\ast}s(x)\ne 0$ and 
the $G$-action on $X_{f^{\ast}s}=f^{-1}(Y_{s})$ is also closed. 
\end{proof}

We will apply the machinery of GIT to  
plane sextic curves (\cite{Sh1}), 
bidegree $(4, 4)$ curves on ${\proj}^{1}\times {\proj}^{1}$ (\cite{Sh2}), and 
point sets in ${\proj}^{2}$ (\cite{GIT}, \cite{D-O}).

\begin{definition}\label{def: simple singularity}
Let $C \subset S$ be a reduced curve on a smooth surface $S$. 
A singular point $p \in C$ is a \textit{simple singularity} if 
$(i)$ $p$ is a double or triple point, and 
$(ii)$ the strict transform of $C$ in the blow-up of $S$ at $p$ does not have triple point over $p$. 
\end{definition}

See \cite{B-H-P-V} II.8 for the A-D-E classification of the simple singularities. 
In this paper we will deal mainly with nodes ($A_1$-points) and 
ordinary triple points ($D_4$-points). 
In some literatures, the condition $(ii)$ above is stated in the form 
``$C$ has no consecutive triple point'' (\cite{Sh1}) or 
``$C$ has no infinitely near triple point'' (\cite{Ho}).

We consider the ${\PGL}$-action on the linear system $|\mathcal{O}_{{\proj}^{2}}(6) |$ of plane sextic curves, 
which is endowed with a natural linearized ample line bundle. 

\begin{proposition}[Shah \cite{Sh1}]\label{stab sextic}
A reduced plane sextic is ${\PGL}$-stable 
if and only if it has only simple singularities. 
\end{proposition}

We also need a stability criterion for the ${\rm PGL}_{2}\times{\rm PGL}_{2}$-action  
on the linear system $|\mathcal{O}_{{\proj}^{1}\times{\proj}^{1}}(4, 4)|$ 
endowed with the naturally linearized $\mathcal{O}(1)$.

\begin{proposition}[Shah \cite{Sh2}]\label{stability of (4, 4) curve}
Let $C \subset {\proj}^{1}\times{\proj}^{1}$ be a reduced curve of bidegree $(4, 4)$.  
If $C$ has only nodes as singularities, 
then $C$ is ${\rm PGL}_{2}\times{\rm PGL}_{2}$-stable. 
\end{proposition}

Finally we consider the diagonal action of ${\PGL}$ on the product $({\proj}^{2})^{d}$. 
Let $U_{d} \subset ({\proj}^{2})^{d}$ be the open set of ordered points $(p_{1}, \cdots, p_{d})$ 
such that no three of $\{ p_{i} \}_{i=1}^{d}$ are collinear,  
and let $V_{d} \subset ({\proj}^{2})^{d}$ be the variety of ordered points $(p_{1}, \cdots, p_{d})$ 
such that $\{ p_{1}, p_{2}, p_{3} \}$ are collinear and no other three of $\{ p_{i} \}_{i=1}^{d}$ are collinear. 

\begin{proposition}[\cite{GIT}, \cite{D-O}]\label{existence of configuration space}
For $d\geq4$ (resp. $d\geq5$) 
a geometric quotient $U_{d}/{\PGL}$ (resp. $V_{d}/{\PGL}$) exists 
and is a quasi-projective rational variety of dimension $2d-8$ (resp. $2d-9$). 
\end{proposition}

\begin{proof}
For the assertion for $U_d$, see \cite{D-O} Chapter II. 
The variety $V_d$ is contained in the stable locus with respect to the ${\SL}$-linearized line bundle 
${\HP} \boxtimes \cdots \boxtimes {\HP}$ so that 
a geometric quotient exists by Theorem \ref{thm: GIT}. 
For $d\geq 7$ the rationality of $V_d/{\PGL}$ follows from the birational equivalence 
$V_{d}/{\PGL} \sim V_{d-4}$. 
The remaining $V_{5}/{\PGL}$ and $V_{6}/{\PGL}$ are also clearly rational. 
\end{proof}


\section{2-elementary $K3$ surfaces}\label{sec: 2-elemen K3}

\subsection{Basic properties}\label{basic}

We recall basic facts on 2-elementary $K3$ surfaces following \cite{Ni2} and \cite{A-N}.  
Let $(X, \iota)$ be a 2-elementary $K3$ surface, i.e., 
a pair of a complex $K3$ surface $X$ and a non-symplectic involution $\iota$ on $X$. 
The surface $X$ is always algebraic due to the presence of $\iota$. 
The invariant and anti-invariant lattices
\begin{equation}\label{inv and anti-inv lattice}
L_{\pm} = L_{\pm}(X, \iota) = \{ l \in H^{2}(X, {\Z}), \ \iota^{\ast}l=\pm l \}
\end{equation}
are even 2-elementary lattices of signature $(1, r-1)$ and $(2, 20-r)$ respectively, 
where $r$ is the rank of $L_+$. 
Note that $L_-$ is the orthogonal complement of $L_+$ in ${\cohomology}$ and hence 
we have a natural isometry $(D_{L_+}, q_{L_+}) \simeq (D_{L_-}, -q_{L_-})$. 
The main invariant $(r, a, \delta)$ of $L_+$ is also called the main invariant of $(X, \iota)$ 
and may be calculated geometrically as follows. 

\begin{proposition}[\cite{Ni2}]\label{compute main inv}
Let $(X, \iota)$ be a 2-elementary $K3$ surface of type $(r, a, \delta)$. 
Let $X^{\iota}$ be the fixed locus of $\iota$. 

$({\rm i})$ If $(r, a, \delta)=(10, 10, 0)$, then $X^{\iota} = \emptyset$. 

$({\rm ii})$ If $(r, a, \delta)=(10, 8, 0)$, then $X^{\iota}$ is a union of two elliptic curves. 

$({\rm iii})$ In other cases, $X^{\iota}$ is decomposed as 
$X^{\iota} = C^{g} \sqcup E_{1} \sqcup \cdots \sqcup E_{k}$ 
where $C^{g}$ is a genus $g$ curve and $E_{1}, \cdots, E_{k}$ are $(-2)$-curves with 
\begin{equation}\label{compute r and a}
g=11-\frac{r+a}{2}, \ \ \ \ \ \ \ k=\frac{r-a}{2}. 
\end{equation}
One has $\delta=0$ if and only if the class of $X^{\iota}$ is divisible by $2$ in $NS_X$. 
\end{proposition}

Let $f\colon X\to Y=X/\langle\iota\rangle$ be the quotient morphism 
and $B=f(X^{\iota})$ be the branch curve of $f$.  
If $X^{\iota}\ne\emptyset$, 
$Y$ is a smooth rational surface and $B$ is a smooth member of $|-\!2K_Y|$. 
Following \cite{A-N}, we call such a pair $(Y, B)$ a \textit{right DPN pair}. 
The 2-elementary $K3$ surface $(X, \iota)$ is recovered from $(Y, B)$ 
as the double cover of $Y$ branched over $B$. 
In this way, 2-elementary $K3$ surfaces with non-empty fixed locus are 
in canonical correspondence with right DPN pairs. 
The invariant $(r, a)$ of $(X, \iota)$ can be read off from the topology of $B$ 
by Proposition \ref{compute main inv}. 
We also have  
\begin{equation}\label{compute r from Y}
r=\rho(Y). 
\end{equation}
For the parity $\delta$, 
if $B=\sum_{i}B_{i}$ is the irreducible decomposition of $B$, 
then we have $\delta=0$ if and only if the class 
$\sum_{i} (-1)^{n_i}[B_i]$ is divisible by $4$ in $NS_Y$ for some $n_i \in \{ 0, 1 \}$. 
The lattice $L_+(X, \iota)$ may be obtained as follows. 

\begin{proposition}\label{generate L_+}
Let $(Y, B)$ be a right DPN pair and $(X, \iota)$ be the associated 2-elementary $K3$ surface 
with the quotient morphism $f\colon X\to Y$. 
Then the invariant lattice $L_+=L_+(X, \iota)$ is generated by the sublattice $f^{\ast}NS_Y$ and 
the classes of irreducible components of $X^{\iota}$. 
\end{proposition}

\begin{proof}
Let $B=\sum_{i=0}^{k}B_i$ be the irreducible decomposition and let $C_i = f^{-1}(B_i)$. 
We have $X^{\iota}=\sum_{i=0}^{k}C_i$ and $C_i  \sim \frac{1}{2}f^{\ast}B_i$. 
According to Kharlamov (\cite{Kh} p.304), 
the relation $\sum_{i=0}^{k}C_i \sim -f^{\ast}K_Y$ is the only relation 
among $\{ C_i \}_{i=0}^{k}$ in $L_+/f^{\ast}NS_Y$.  
Since the lattice $f^{\ast}NS_Y \simeq NS_Y(2)$ 
is of index $2^{\frac{1}{2}(r-a)}=2^{k}$ in $L_+$, 
this proves the assertion. 
\end{proof}

\subsection{Right resolutions of plane sextics}\label{right resol}

We explain a relationship between 2-elementary $K3$ surfaces and 
plane sextics with only simple singularities. 
The topic is classical as it goes back to Horikawa \cite{Ho} and Shah \cite{Sh1}. 
Here we develop the argument in more generality in the framework of Alexeev-Nikulin \cite{A-N}. 
Recall from \cite{A-N}
that a \textit{DPN pair} is a pair $(Y, B)$ of 
a smooth rational surface $Y$ and an anti-bicanonical curve $B \in |\!-\!2K_Y|$ 
with only simple singularities.

\begin{definition}\label{def: right resol}
A \textit{right resolution} of a DPN pair $(Y_0, B_0)$ is a triplet $(Y, B, \pi)$ 
such that $(Y, B)$ is a right DPN pair and $\pi\colon Y \to Y_0$ is a birational morphism with $\pi(B)=B_0$. 
By abuse of terminology, we also call $(Y, B, \pi)$ a right resolution of $B_0$ 
when $Y_0$ is obvious from the context. 
\end{definition}

\begin{proposition}[cf. \cite{A-N}]\label{existence of right resol}
A right resolution of a DPN pair $(Y_0, B_0)$ exists and is unique up to isomorphism. 
\end{proposition}

\begin{proof}
Let $S\to Y_0$ be the double cover branched over $B_0$. 
As $B_0$ has only simple singularities, 
$S$ is a normal surface with only A-D-E singularities (corresponding to those of $B_0$) 
and with trivial canonical divisor. 
The minimal resolution $X$ of $S$ is a $K3$ surface, 
and the covering transformation of $S \to Y_{0}$ 
induces a non-symplectic involution $\iota$ on $X$. 
If $(Y, B)$ is the right DPN pair associated to $(X, \iota)$,  
then by the universality of the quotient $X \to Y$ 
we have a birational morphism $\pi\colon Y \to Y_0$ with $\pi(B)=B_0$. 
This proves the existence. 
For any other right resolution $(Y', B', \pi')$ with the associated 2-elementary $K3$ surface $(X', \iota')$, 
let $X'\to S' \to Y_0$ be the Stein factorization of the morphism $X' \to Y' \to Y_0$. 
Then $S' \to Y_0$ is a double cover branched over $B_0$ and thus is isomorphic to $S \to Y_0$. 
It follows that $X'\to Y_0$ is isomorphic to $X\to Y_0$ and we have $(Y, B, \pi)\simeq (Y', B', \pi')$. 
\end{proof}

In \cite{A-N} right resolution is constructed explicitly as follows. 
Let 
\begin{equation}\label{right resol prcs} 
\cdots \stackrel{\pi_{i+1}}{\to} (Y_{i}, B_{i}) \stackrel{\pi_{i}}{\to} (Y_{i-1}, B_{i-1}) \stackrel{\pi_{i-1}}{\to} \cdots 
\stackrel{\pi_{1}}{\to} (Y_{0}, B_{0})  
\end{equation} 
be the successive blow-ups defined inductively by 
\begin{equation}\label{right resol prcs II}
Y_{i+1} = {\rm bl}_{\Sigma_i}Y_{i},  \ \ \ 
B_{i+1} = \widetilde{B_{i}} + \sum_{k=1}^{N} (m_{k}-2)E_{k}, 
\end{equation}
where 
$\Sigma_i = \{ p_{k} \}_{k=1}^{N}$ is the singular locus of $B_{i}$, 
$\widetilde{B_{i}}$ is the strict transform of $B_{i}$, 
$E_{k}$ is the $(-1)$-curve over $p_{k}$, 
and $m_k$ is the multiplicity of $B_{i}$ at $p_{k}$. 
Each $(Y_i, B_i)$ is also a DPN pair. 
This process will terminate and we finally obtain a right DPN pair $(Y, B)$. 

In this way, one can associate a 2-elementary $K3$ surface $(X, \iota)$ to a DPN pair $(Y_0, B_0)$ 
by taking its right resolution $(Y, B, \pi)$. 
Composing $\pi$ with the quotient map $X\to Y$, 
we have a natural generically two-to-one morphism $g\colon X\to Y_0$ branched over $B_0$. 
In this paper we will deal only with the following simple situations. 

\begin{example}\label{ex:1}
When $B_0$ has only nodes $p_1,\cdots, p_a$ as the singularities, 
then $E_i=g^{-1}(p_i)$ is a $(-2)$-curve on $X$, 
and each component of the fixed curve $X^{\iota}$ is mapped by $g$ 
birationally onto a component of $B_0$. 
By Proposition \ref{generate L_+} the lattice $L_+(X, \iota)$ is generated by 
the sublattice $g^{\ast}NS_{Y_0}\simeq M_{\rho(Y_0)}$, 
the classes of the $(-2)$-curves $E_1, \cdots, E_a$, and of the components of $X^{\iota}$. 
In particular, we have $r=\rho(Y_0)+a$. 
\end{example} 

\begin{example}\label{ex:2}
As a slight generalization, 
suppose that ${\rm Sing}(B_0)$ consists of nodes $p_1,\cdots, p_a$ and ordinary triple points $q_1,\cdots, q_d$. 
Then the curve $g^{-1}(q_j)$ is decomposed as 
$g^{-1}(q_j)=G_j+\sum_{k=1}^3E_{jk}$ such that 
$G_j$ is a rational component of $X^{\iota}$, and 
$E_{jk}$ are the $(-2)$-curves over the infinitely near points of $q_j$ given by the branches of $B_0$. 
We have $(G_j.E_{jk})=1$ and $(E_{jk}.E_{jk'})=-2\delta_{kk'}$. 
Other components of $X^{\iota}$ than $G_1,\cdots, G_d$ are mapped by $g$ 
birationally onto the components of $B_0$. 
The lattice $L_+(X, \iota)$ is generated by $g^{\ast}NS_{Y_0}$,   
the classes of the $(-2)$-curves $g^{-1}(p_i)$, $E_{jk}$, $G_j$, and of those components of $X^{\iota}$. 
In particular, we have $r=\rho(Y_0)+a+4d$. 
\end{example}

When $Y_0={\proj}^2$ or ${\proj}^1\times{\proj}^1$, 
for which $B_0$ is a sextic or a bidegree $(4, 4)$ curve respectively, 
we have the following useful property.

\begin{lemma}\label{proj mor}
Let $(Y_0, B_0)$ be a DPN pair with $Y_0$ being either ${\proj}^2$ or a smooth quadric in ${\proj}^3$. 
Let $(X, \iota)$ be the associated 2-elementary $K3$ surface with the natural projection $g:X\to Y_0$.  
Then the morphism $g\colon X\to Y_0 \subset {\proj}^d$ can be identified with the morphism 
$\phi_H\colon X \to |H|^{\vee}$ associated to the bundle $H = g^{\ast}\mathcal{O}_{Y_0}(1)$. 
\end{lemma}

\begin{proof}
The bundle $H$ is nef and big. 
Use the Riemann-Roch formula and the vanishing $h^i(H)=0$ for $i>0$ 
to see that $|H|=g^{\ast}|\mathcal{O}_{Y_0}(1)|$. 
\end{proof}

\subsection{Classification and the moduli spaces}\label{subsec: moduli}

2-elementary $K3$ surfaces were classified by Nikulin in terms of the main invariants. 

\begin{theorem}[Nikulin \cite{Ni2}]\label{classify 2-ele K3}
The deformation type of a 2-elementary $K3$ surface $(X, \iota)$ is determined by 
the main invariant $(r, a, \delta)$. 
All possible main invariants of 2-elementary $K3$ surfaces are shown on the following Figure \ref{Nikulin table} which is 
identical to the table in page 31 of \cite{A-N}. 
\begin{figure}[h]
\centerline{\includegraphics[width=10cm]{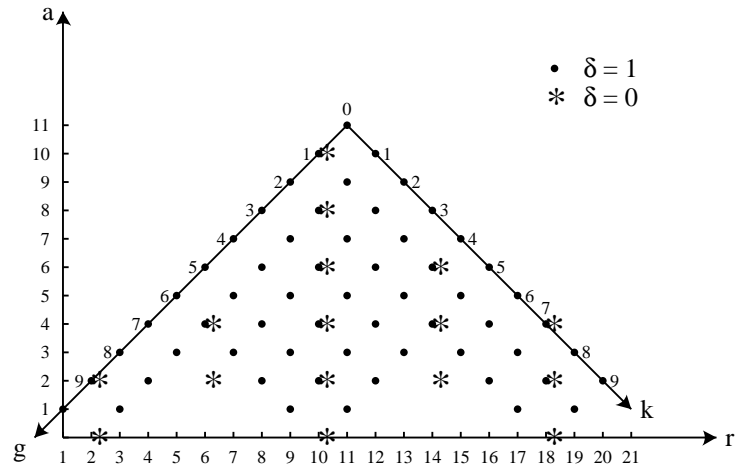}} 
\caption{Geography of main invariants $(r,a,\delta)$} 
\label{Nikulin table}
\end{figure}
\end{theorem}

A moduli space of 2-elementary $K3$ surfaces of main invariant $(r, a, \delta)$ 
is constructed as follows. 
We fix an even 2-elementary lattice $L$ of main invariant $(2, 20-r, a, \delta)$, 
which is isometric to the anti-invariant lattice of every 2-elementary $K3$ surface 
of type $(r, a, \delta)$. 
Let $\mathcal{F}({\rm O}(L)^+) = {\rm O}(L)^+\backslash\Omega_L^+$ 
be the modular variety associated to ${\rm O}(L)^+$. 
The divisor $\sum \delta^{\perp} \subset \Omega_L^+$, 
where $\delta$ are $(-2)$-vectors in $L$, 
is the inverse image of an algebraic divisor $D\subset\mathcal{F}({\rm O}(L)^+)$. 
Let ${\moduli}$ be the variety 
\begin{equation}\label{def of moduli}
{\moduli} = \mathcal{F}({\rm O}(L)^{+}) - D, 
\end{equation}
which is normal, irreducible, quasi-projective, and of dimension $20-r$. 
For a 2-elementary $K3$ surface $(X, \iota)$ of type $(r, a, \delta)$,  
we may choose an isometry $\Phi \colon L_-(X, \iota) \to L$ with $\Phi(H^{2,0}(X)) \in \Omega_{L}^{+}$. 
Then we define the period of $(X, \iota)$ by  
\begin{equation}\label{period of 2-ele K3}
\mathcal{P}(X, \iota) = [\Phi(H^{2,0}(X))] \in {\moduli}, 
\end{equation}
which is independent of the choice of $\Phi$. 

\begin{theorem}[Yoshikawa \cite{Yo1}, \cite{Yo3}]\label{thm: moduli}
The variety ${\moduli}$ is a moduli space of 2-elementary $K3$ surfaces of type $(r, a, \delta)$ 
in the following sense. 

$({\rm i})$ For a complex analytic family $(\frak{X}\to U, \iota)$ of such 2-elementary $K3$ surfaces, 
the period map $\mathcal{P}\colon U\to{\moduli}, u\mapsto\mathcal{P}(\frak{X}_u, \iota _u)$, is holomorphic. 
When the family is algebraic, $\mathcal{P}$ is a morphism of algebraic varieties. 

$({\rm i})$ Via the period mapping, the points of ${\moduli}$ are in one-to-one correspondence with 
the isomorphism classes of 2-elementary $K3$ surfaces of type $(r, a, \delta)$. 
\end{theorem}

\subsection{The discriminant covers}\label{ssec:discri cover}

Let $L$ be the lattice used in the definition \eqref{def of moduli} 
and ${\cover}$ be the modular variety 
\begin{equation}\label{def: discri cover}
{\cover} = \mathcal{F}(\widetilde{{\rm O}}(L)^{+}),  
\end{equation}
which is a Galois cover of $\mathcal{F}({\rm O}(L)^{+})$ with the Galois group ${\rm O}(D_L, q_L)$. 
We call ${\cover}$ the \textit{discriminant cover} of ${\moduli}$. 
Since $\widetilde{{\rm O}}(L)^+\ne\widetilde{{\rm O}}(L)$, 
we may identify ${\cover} = \widetilde{{\rm O}}(L)\backslash \Omega_L$. 
The next proposition is a key for our proof of Theorem \ref{main}. 

\begin{proposition}\label{isogeny bet discri cover}
Let $(r, a, \delta)$ and $(r, a', \delta')$ be main invariants of 2-elementary $K3$ surfaces. 
Assume that either $({\rm i})$ $\delta =1, a>a'$, or $({\rm ii})$ $\delta =\delta', a>a'$. 
Then one has a finite surjective morphism 
$\varphi\colon{\cover}\to\widetilde{\mathcal{M}}_{(r,a',\delta')}$. 
\end{proposition}

\begin{proof}
Let $L$ and $L'$ be even 2-elementary lattices of main invariant $(2, 20-r, a, \delta)$ 
and $(2, 20-r, a', \delta')$ respectively. 
Calculating the discriminant form $(D_L, q_L)$ explicitly, 
one can find an isotropic subgroup $G \subset D_L$ such that 
the 2-elementary quadratic form $(G^{\perp}/G, q_L)$ has the invariant $(a', \delta')$. 
By the coincidence of main invariant, 
the overlattice of $L$ defined by $G$ is isometric to $L'$. 
Hence the assertion follows from Proposition \ref{construct isogeny}. 
\end{proof}

The relationship between the modular varieties is as follows. 
\begin{equation}\label{relation bet modular varieties}
\begin{CD}
{\cover}-H  @>{\varphi}>>  \widetilde{\mathcal{M}}_{(r, a', \delta')}-H'  \\ 
@VVV                                          @VVV                                     \\
{\moduli}    @.                      \mathcal{M}_{(r, a', \delta')}       
\end{CD}
\end{equation}
Here $H$ and $H'$ are appropriate Heegner divisors.

\begin{remark}\label{interpret by twisted FM partner}
When $a'=a-2$, $\varphi$ admits the following geometric interpretation. 
For an $\omega\in{\cover}$ let $(X, \iota)\in {\moduli}$ and $(X', \iota')\in\mathcal{M}_{(r,a',\delta')}$ 
be the 2-elementary $K3$ surfaces given by the images of $\omega$ and $\varphi(\omega)$ respectively. 
Then $X$ is derived equivalent to the twisted $K3$ surface $(X', \alpha')$ 
for a Brauer element $\alpha'\in{\rm Br}(X')$ of order $\leq2$. 
Indeed, we have a Hodge embedding $T_X \hookrightarrow T_{X'}$ of the transcendental lattices of index $\leq2$ 
so that the twisted derived Torelli theorem \cite{H-S} applies. 
\end{remark}

General points of ${\cover}$ may be obtained as follows (cf. \cite{Do}, \cite{A-N}). 
We fix an even hyperbolic 2-elementary lattice $L_+$ of main invariant $(r, a, \delta)$, 
a primitive embedding $L_+\subset\Lambda_{K3}$, 
and an isometry $(L_+)^{\perp}\cap\Lambda_{K3}\to L$.  
Let $(X, \iota)\in{\moduli}$ and $j\colon L_+\to L_+(X, \iota)$ be a given isometry. 
By Proposition \ref{unique emb} the isometry $j$ can be extended to an isometry 
$\Phi\colon \Lambda_{K3}\to{\cohomology}$, 
which in turn induces the isometry $\Phi|_L\colon L\to L_-(X, \iota)$. 
By Proposition \ref{glue} the isometry $\Phi|_L$ is determined from $j$ up to the action of $\widetilde{{\rm O}}(L)$. 
Then we define the period of the lattice-marked 2-elementary $K3$ surface $((X, \iota), j)$ by 
\begin{equation}\label{lifted period}
\widetilde{\mathcal{P}}((X, \iota), j) = [\Phi|_L^{-1}(H^{2,0}(X))] \in {\cover}. 
\end{equation}
If we define equivalence of two such objects $((X, \iota), j)$ and $((X', \iota'), j')$ by 
the existence of a Hodge isometry $\Psi\colon{\cohomology}\to H^2(X', {\Z})$ with $j'=\Psi\circ j$, 
then via the period mapping $\widetilde{\mathcal{P}}$ the open set of ${\cover}$ over ${\moduli}$ 
parametrizes the equivalence classes of such objects $((X, \iota), j)$. 
The assignment $((X, \iota), j)\mapsto(X, \iota)$ gives the projection ${\cover}\dashrightarrow{\moduli}$. 

This interpretation of ${\cover}$ using lattice-marked 2-elementary $K3$ surfaces is useful, but not so geometric. 
In the rest of this paper, using this interpretation intermediately, 
we will seek for more geometric interpretations for some of ${\cover}$.

Here is a general strategy. 
We define a space $U$ parametrizing certain plane sextics $B$ 
(or bidegree $(4, 4)$ curves on ${\proj}^1\times{\proj}^1$) 
which are endowed with some labeling of their singularities and components. 
The 2-elementary $K3$ surface $(X, \iota)$ associated to 
the right resolution of $B$ has main invariant $(r, a, \delta)$. 
The point is that the labeling for $B$ induces an isometry $j\colon L_+\to L_+(X, \iota)$. 
Actually, an argument as in Examples \ref{ex:1} and \ref{ex:2} 
will suggest an appropriate definition of the reference lattice $L_+$, 
and then $j$ will be obtained naturally. 
Considering the period of $((X, \iota), j)$ as defined above, 
we obtain a morphism $p\colon U\to{\cover}$. 
We will prove that $p$ descends to an open immersion $U/G\to{\cover}$ 
where $G={\PGL}$ (or ${\rm PGL}_2\times{\rm PGL}_2$). 
This amounts to showing that ${\dim}(U/G)=20-r$ and that the $p$-fibers are $G$-orbits. 
The latter property is verified using the Torelli theorem and that 
the curve $B$ with its labeling may be recovered from $((X, \iota), j)$ via Lemma \ref{proj mor}.

In this way, some of ${\cover}$ will be birationally identified with 
the moduli of certain curves with labeling. 
Such geometric interpretations vary according to ${\cover}$, and are out of single formulation. 
However, the processes by which we attach them to ${\cover}$ are largely common, as suggested above.  
Then, in order to avoid repetition, 
we will discuss such processes in detail for only few cases 
(Section \ref{ssec: (r,r,1)}). 
For other cases, we omit some detail and refer to Section \ref{ssec: (r,r,1)} as a model.

Now our geometric descriptions will imply that those ${\cover}$ are often unirational. 
With the aid of Proposition \ref{isogeny bet discri cover}, 
we will then obtain the following. 

\begin{theorem}\label{unirat disc cover}
The discriminant covers ${\cover}$ are unirational except possibly for 
$(r, a) = (10, 10),  (11, 11),  (12, 10),  (13, 9)$.
\end{theorem}

Sometimes our interpretations of ${\cover}$ using sextics are translated into yet another geometric interpretations, 
such as configuration spaces of points in ${\proj}^2$.


\section{The case $r\leq9$}\label{sec: r <10}

In this section we prove that ${\cover}$ are unirational for $r\leq9$.  
We first prove in Section \ref{ssec: (r,r,1)} the unirationality of $\widetilde{\mathcal{M}}_{(r,r,1)}$ 
with $r\leq9$ using the Severi varieties of nodal plane sextics. 
These cases are model for the subsequent sections and hence discussed in detail. 
From Proposition \ref{isogeny bet discri cover} and Figure \ref{Nikulin table} 
follows the unirationality of ${\cover}$ with $r\leq9$ and $(r, a, \delta)\ne(2, 2, 0)$. 
In Section \ref{ssec: (2,2,0)} we treat $\widetilde{\mathcal{M}}_{(2,2,0)}$.


\subsection{$\widetilde{\mathcal{M}}_{(r,r,1)}$ and the Severi varieties of nodal sextics}\label{ssec: (r,r,1)}

For $r\leq11$ let $V_{r-1}\subset|{\sextic}|$ be the variety of irreducible plane sextics  
with $r-1$ nodes and with no other singularity. 
The variety $V_{r-1}$, known as a \textit{Severi variety}, is  
smooth, of dimension $28-r$, and irreducible (\cite{Ha}). 
By endowing the sextics with markings of the nodes, 
we have the following $\frak{S}_{r-1}$-cover of $V_{r-1}$: 
\begin{equation}\label{marked Severi}
\widetilde{V}_{r-1} = 
   \{ \ (C, p_{1}, \cdots, p_{r-1})\in V_{r-1}\times ({\proj}^{2})^{r-1}, \ \ {\rm Sing}(C)= \{ p_{i} \} _{i=1}^{r-1} \: \: \} . 
\end{equation}
By Lemma \ref{stab & fin mor} and Proposition \ref{stab sextic} 
we have a geometric quotient $\widetilde{V}_{r-1}/{\PGL}$.  

\begin{proposition}\label{rationality of V_<10}
For $r\leq9$ the variety $\widetilde{V}_{r-1}$ is rational. 
In particular, the quotient $\widetilde{V}_{r-1}/{\PGL}$ is a unirational variety of dimension $20-r$. 
\end{proposition}

\begin{proof}
We consider the nodal map 
\begin{equation}\label{nodal map}
\kappa : \widetilde{V}_{r-1} \to ({\proj}^{2})^{r-1}, \ \ \ (C, p_{1}, \cdots, p_{r-1}) \mapsto (p_{1}, \cdots, p_{r-1}). 
\end{equation}
For a general $\mathbf{p} = (p_{1}, \cdots, p_{r-1})$ 
the fiber $\kappa^{-1}(\mathbf{p})$ may be identified with an open set of $|\!-\!2K_Y|$ 
where $Y$ is the blow-up of ${\proj}^{2}$ at $\{ p_{i} \} _{i=1}^{r-1}$. 
Since $Y$ is a del Pezzo surface, 
we have ${\dim} |\!-\!2K_Y| \geq 3$ so that $\kappa$ is dominant. 
As $\kappa^{-1}(\mathbf{p})$ is an open set of a linear subspace of $|\mathcal{O}_{{\proj}^{2}}(6)|$,  
we see that $\widetilde{V}_{r-1}$ is birationally equivalent to the projective bundle associated to a locally free sheaf 
on an open set of $({\proj}^{2})^{r-1}$. 
\end{proof}

We shall construct a period map $\tilde{p}\colon\widetilde{V}_{r-1}\to\widetilde{\mathcal{M}}_{(r,r,1)}$ for $r\leq11$. 
For a sextic with labeling $(C, \mathbf{p})=(C, p_1, \cdots, p_{r-1})$ in $\widetilde{V}_{r-1}$, 
let $(X, \iota)$ be the 2-elementary $K3$ surface associated to the right resolution of $C$, 
and $g\colon X \to {\proj}^2$ be the natural projection branched over $C$. 
The quotient $X/\langle\iota\rangle$ is the blow-up of ${\proj}^2$ at $p_1,\cdots, p_{r-1}$. 
On $X$ we have the line bundle $H=g^{\ast}{\HP}$ and the $(-2)$-curves $E_i=g^{-1}(p_i)$. 
Let $M_r=\langle h, e_1, \cdots , e_{r-1}\rangle$ be the lattice defined in \eqref{Cremona lattice}.  
By Example \ref{ex:1}, the classes of $H$ and $E_1, \dots, E_{r-1}$ define an isometry of lattices 
$j\colon M_r \to L_+(X, \iota)$ by $h\mapsto[H]$ and $e_i\mapsto[E_i]$. 
We thus associate a lattice-marked 2-elementary $K3$ surface $((X, \iota), j)$ to $(C, \mathbf{p})$. 
Fixing a primitive embedding $M_r\hookrightarrow\Lambda_{K3}$ and 
considering the period of $((X, \iota), j)$ as defined in \eqref{lifted period}, 
we then obtain a point $\tilde{p}(C, \mathbf{p})$ in $\widetilde{\mathcal{M}}_{(r, r, 1)}$.

\begin{proposition}\label{Torelli r<12}
Let $r\leq11$. 
Two sextics with labeling $(C, \mathbf{p}), (C', \mathbf{p}') \in \widetilde{V}_{r-1}$ are ${\PGL}$-equivalent 
if and only if $\tilde{p}(C, \mathbf{p}) = \tilde{p}(C', \mathbf{p}')$. 
\end{proposition}

\begin{proof}
It suffices to prove the ``if'' part. 
Let $X, j, H,\cdots$ (resp. $X', j', H',\cdots$) be the objects 
constructed from $(C, \mathbf{p})$ (resp. $(C', \mathbf{p}')$) as above. 
If $\tilde{p}(C, \mathbf{p}) = \tilde{p}(C', \mathbf{p}')$, 
we have a Hodge isometry $\Phi\colon H^2(X', {\Z}) \to {\cohomology}$ with $j=\Phi\circ j'$. 
This equality means that $\Phi([H'])=[H]$ and $\Phi([E_{i}'])=[E_{i}]$. 
Since $\Phi$ maps the ample class $4H'-\sum_{i=1}^{r-1}E_{i}'$ to the ample class $4H-\sum_{i=1}^{r-1}E_{i}$, 
by the strong Torelli theorem there exists an isomorphism $\varphi\colon X\to X'$ with $\varphi^{\ast}=\Phi$. 
Then we have $\varphi(E_{i})=E_{i}'$ and $\varphi^{\ast}H' \simeq H$. 
By Lemma \ref{proj mor} we obtain an automorphism $\psi\colon{\proj}^2\to{\proj}^2$ 
with $g'\circ\varphi=\psi \circ g$. 
Since $p_i=g(E_i)$ and $p_i'=g'(E_i')$, we have $\psi(p_i)=p_i'$. 
Since $C$ and $C'$ are respectively the branches of $g$ and $g'$, we also have $\psi(C)=C'$. 
\end{proof}

\begin{theorem}\label{period map r<12}
Let $r\leq11$. 
The period map $\tilde{p}\colon\widetilde{V}_{r-1} \to \widetilde{\mathcal{M}}_{(r,r,1)}$ is a morphism of varieties 
and induces an open immersion $\widetilde{V}_{r-1}/{\PGL} \to \widetilde{\mathcal{M}}_{(r,r,1)}$. 
\end{theorem}

\begin{proof}
We repeat the above construction for families. 
Let $\widetilde{\mathcal{C}}_{r-1} \subset \widetilde{V}_{r-1} \times {\proj}^2$ be 
the universal marked nodal sextic over $\widetilde{V}_{r-1}$ 
(which may be obtained from the universal sextic over $V_{r-1}$). 
We have the sections $s_i\colon\widetilde{V}_{r-1} \to \widetilde{\mathcal{C}}_{r-1}$ defined by 
$(C, \mathbf{p}) \mapsto ((C, \mathbf{p}), p_i)$ where $\mathbf{p}=(p_1, \cdots, p_{r-1})$. 
There is an open set $\widetilde{V} \subset \widetilde{V}_{r-1}$ such that the divisor 
$\widetilde{\mathcal{C}} = \widetilde{\mathcal{C}}_{r-1}\vert_{\widetilde{V}}$ 
of $\widetilde{V}\times {\proj}^2$ is linearly equivalent to $\pi_2^{\ast}{\sextic}$ 
where $\pi_2\colon\widetilde{V}\times {\proj}^2\to{\proj}^2$ is the projection. 
We denote $W_i = s_{i}(\widetilde{V})$. 
Let $\mathcal{Y}$ be the blow-up of $\widetilde{V} \times {\proj}^2$ along $\mathop{\bigcup}_{i=1}^{r-1} W_i$ 
and $\mathcal{D}_i \subset \mathcal{Y}$ be the exceptional divisor over $W_i$. 
Since the strict transform $\mathcal{B} \subset \mathcal{Y}$ of $\widetilde{\mathcal{C}}$  
is a smooth divisor linearly equivalent to 
$\pi_{2}^{\ast}{\sextic}-2 \sum_{i=1}^{r-1}\mathcal{D}_i$, 
we may take a double cover $f\colon\mathcal{X}\to\mathcal{Y}$ branched over $\mathcal{B}$. 
The natural projection $\pi\colon\mathcal{X}\to\widetilde{V}$ is a family of $K3$ surfaces. 
Let $\iota$ be the covering transformation of $f$ and 
$\mathcal{L}_+$ be the local system $(R^2\pi_{\ast}{\Z})^{\iota}$ over $\widetilde{V}$. 
Then the divisors $\{ f^{-1}(\mathcal{D}_i) \} _i$ and the pullback of $\pi_{2}^{\ast}{\HP}$ 
define a trivialization $\mathcal{L}_+ \to M_r \times \widetilde{V}$. 
This means that the monodromy group of the local system 
$\mathcal{L}_- = (\mathcal{L}_+)^{\perp}\cap R^2\pi_{\ast}{\Z}$ 
is contained in $\widetilde{{\rm O}}(L_r)$ where $L_r=(M_r)^{\perp}\cap\Lambda_{K3}$. 
Considering the local system $\mathcal{L}_-$, 
we see that the period map 
$\tilde{p}|_{\widetilde{V}}\colon\widetilde{V}\to\widetilde{\mathcal{M}}_{(r,r,1)}$
is a locally liftable holomorphic map. 
By Borel's extension theorem \cite{Bo}, 
$\tilde{p}|_{\widetilde{V}}$ is a morphism of algebraic varieties. 
This implies that $\tilde{p}$ is a morphism of varieties. 
By the ${\PGL}$-invariance $\tilde{p}$ induces a morphism 
$\widetilde{\mathcal{P}}\colon\widetilde{V}_{r-1}/{\PGL}\to\widetilde{\mathcal{M}}_{(r,r,1)}$. 
Proposition \ref{Torelli r<12} implies the injectivity of $\widetilde{\mathcal{P}}$. 
Then $\widetilde{\mathcal{P}}$ is dominant because    
we have $\dim (\widetilde{V}_{r-1}/{\PGL}) = 20-r$ and 
$\widetilde{\mathcal{M}}_{(r,r,1)}$ is irreducible. 
Thus $\widetilde{\mathcal{P}}$ is an open immersion by the Zariski's Main Theorem. 
\end{proof}

\begin{corollary}\label{unirat for r<10}
If $r\leq9$ and $(r, a, \delta)\ne (2, 2, 0)$, then ${\cover}$ is unirational. 
\end{corollary}

\begin{proof}
By Proposition \ref{rationality of V_<10} and Theorem \ref{period map r<12}, 
$\widetilde{\mathcal{M}}_{(r, r, 1)}$ is unirational for $r\leq9$. 
Then the assertion follows from Proposition \ref{isogeny bet discri cover} and Figure \ref{Nikulin table}. 
\end{proof}

\begin{remark}\label{after Morrison-Saito}
Morrison-Saito \cite{M-S} constructed an open immersion $V_{r-1}/{\PGL} \to \mathcal{F}(\Gamma_r)$ 
for a certain arithmetic group $\Gamma_r\subset{\rm O}(L_r)^+$. 
Our idea to relate $\widetilde{\mathcal{M}}_{(r,r,1)}$ with $\widetilde{V}_{r-1}$  
was inspired by their argument. 
\end{remark}

\begin{remark}\label{rationality of mkd Severi}
In fact, $\widetilde{V}_{r-1}/{\PGL}$ is rational when $2\leq r\leq 9$. 
For $r\geq5$ this may be seen by fixing first four nodes in general position. 
For $r\leq 4$ we need invariant-theoretic techniques. 
In the rest of the paper, one would find that several ${\cover}$ are rational as well.  %
\end{remark}


\subsection{$\widetilde{\mathcal{M}}_{(2,2,0)}$ and bidegree $(4, 4)$ curves}\label{ssec: (2,2,0)}

Let $Q = {\proj}^1\times{\proj}^1$ be a smooth quadric embedded in ${\proj}^3$. 
The group $G = {\rm PGL}_2\times{\rm PGL}_2$ acts naturally on $Q$. 
Let $U \subset |\mathcal{O}_{Q}(4, 4)|$ be the open set of smooth bidegree $(4, 4)$ curves. 
By Proposition \ref{stability of (4, 4) curve} we have a geometric quotient $U/G$ 
as an affine unirational variety of dimension $18$.

For a curve $C\in U$  
let $(X, \iota)$ be the 2-elementary $K3$ surface associated to the right DPN pair $(Q, C)$ 
and $f\colon X\to Q$ be the quotient morphism. 
The lattice $L_+(X, \iota)$ is equal to $f^{\ast}NS_Q$ by Proposition \ref{generate L_+}, 
and thus isometric to the lattice $U(2)$. 
In fact, using the basis $\{ u, v\}$ of $U(2)$ defined in \eqref{U(2)}, 
we have an isometry $j\colon U(2) \to L_+(X, \iota)$ by 
$u\mapsto [f^{\ast}\mathcal{O}_{Q}(1, 0)]$ and $v\mapsto [f^{\ast}\mathcal{O}_{Q}(0, 1)]$. 
Here it is important to distinguish the two rulings of $Q$. 
In this way, we obtain a lattice-marked 2-elementary $K3$ surface $((X, \iota), j)$ from $C$. 
We then obtain a point $\tilde{p}(C)$ in $\widetilde{\mathcal{M}}_{(2,2,0)}$ 
as the period of $((X, \iota), j)$ as before.

In this construction, one may recover the morphism $f\colon X\to Q$ (and hence its branch $C$) 
from the class $j(u+v)$ by Lemma \ref{proj mor}. 
By using $f$, the two rulings $|\mathcal{O}_Q(1, 0)|$, $|\mathcal{O}_Q(0, 1)|$ of $Q$ 
may be respectively recovered from the elliptic fibrations on $X$ 
given by the classes $j(u)$, $j(v)$. 

\begin{theorem}\label{period map (2,2,0)}
The period map $\tilde{p}\colon U\to\widetilde{\mathcal{M}}_{(2,2,0)}$ is a morphism of varieties 
and induces an open immersion $U/G\to\widetilde{\mathcal{M}}_{(2,2,0)}$. 
In particular, $\widetilde{\mathcal{M}}_{(2,2,0)}$ is unirational. 
\end{theorem}

\begin{proof}
Basically one may apply a similar argument as for 
Proposition \ref{Torelli r<12} and Theorem \ref{period map r<12}. 
In the present case, one should note that $G$ is the group of 
automorphisms of $Q$ preserving the two rulings respectively. 
This ensures the $G$-invariance of $\tilde{p}$ 
for its definition involves the distinction of the two rulings. 
The recovery of the morphisms $f$, the curves $C$, and the two rulings of $Q$ as explained above 
implies the injectivity of the induced morphism $U/G\to\widetilde{\mathcal{M}}_{(2,2,0)}$. 
Here one may apply the strong Torelli theorem by using the ample classes $j(u+v)$. 
\end{proof}


\section{The case $r=10$}\label{sec: r=10}

In this section we prove that $\mathcal{M}_{(10,a,\delta)}$ are unirational. 
Kond\=o \cite{Ko1} proved the rationality of $\mathcal{M}_{(10,10,0)}$, the moduli of Enriques surfaces, 
and of $\mathcal{M}_{(10,2,0)}$. 
We study the remaining moduli spaces. 
In Sections \ref{ssec: (10,8,0)} and \ref{ssec: (10,8,1)} we prove the unirationality of 
$\widetilde{\mathcal{M}}_{(10,8,0)}$ and $\widetilde{\mathcal{M}}_{(10,8,1)}$ respectively, 
which implies that $\widetilde{\mathcal{M}}_{(10,a,\delta)}$ are unirational for $a\leq8$. 
The unirationality of $\mathcal{M}_{(10,10,1)}$ is proved in Section \ref{ssec: (10,10,1)}.


\subsection{$\widetilde{\mathcal{M}}_{(10,8,0)}$ and cubic pairs}\label{ssec: (10,8,0)}

Let $U \subset |{\sextic}|\times({\proj}^2)^8$ be the space of pointed sextics 
$(C_1+C_2, \mathbf{p})=(C_1+C_2, p_1,\cdots, p_8)$ 
such that $C_1$ and $C_2$ are smooth cubics transverse to each other 
and that $p_1,\cdots, p_8$ are distinct points contained in $C_1 \cap C_2$. 
The variety $U$ is unirational. 
Indeed, if we denote by 
$V\subset|{\cubic}|\times({\proj}^2)^8$ the locus of $(C, p_1,\cdots, p_8)$ such that $\{ p_i \}_{i=1}^{8}\subset C$, 
then $U$ is dominated by the fiber product $V \mathop{\times}_{({\proj}^2)^8} V$. 
As the projection $V \to ({\proj}^2)^8$ is dominant with a general fiber being a line in $|{\cubic}|$, 
the variety $V \mathop{\times}_{({\proj}^2)^{8}} V$ is rational,  
and so $U$ is unirational. 
By Proposition \ref{stab sextic} and Lemma \ref{stab & fin mor}, 
the natural projection $U\to|{\sextic}|$ shows that 
we have a geometric quotient $U/{\PGL}$ as a unirational variety of dimension $10$.

For a pointed sextic $(C_1+C_2, \mathbf{p}) \in U$ 
we denote by $p_9$ the ninth intersection point of $C_1$ and $C_2$. 
This gives a complete labeling of the nodes of $C_1+C_2$. 
Let $(X, \iota)$ be the 2-elementary $K3$ surface associated to $C_1+C_2$ 
and $g\colon X\to {\proj}^2$ be the natural projection branched over $C_1+C_2$. 
The quotient $X/\langle\iota\rangle$ is the blow-up of ${\proj}^2$ at $p_1,\cdots, p_9$, 
and is a rational elliptic surface. 
We have the decomposition $X^{\iota}=F_1+F_2$ such that $g(F_i)=C_i$.  
By Example \ref{ex:1}, the lattice $L_+(X, \iota)$ is generated by the classes of 
the bundle $H=g^{\ast}{\HP}$, the $(-2)$-curves $E_i=g^{-1}(p_i)$ for $i\leq9$, 
and the elliptic curves $F_1\sim F_2$. 
This suggests to define a reference lattice $L_+$ as follows. 
Let $M_{10}=\langle h, e_1,\cdots, e_9\rangle$ be the lattice defined in \eqref{Cremona lattice} 
and $v\in M_{10}^{\vee}$ be the vector defined by $2v=3h-\sum_{i=1}^{9}e_i$. 
The even overlattice $L_+=\langle M_{10}, v\rangle$ is 2-elementary of main invariant $(10, 8, 0)$. 
Then we have a natural isometry $j\colon L_+\to L_+(X, \iota)$ by sending 
$h\mapsto[H], e_i\mapsto[E_i]$, and $v\mapsto[F_j]$. 
Therefore we obtain a point $\tilde{p}(C_1+C_2, \mathbf{p})$ in $\widetilde{\mathcal{M}}_{(10,8,0)}$ 
as the period of $((X, \iota), j)$ as before.

As in Section \ref{ssec: (r,r,1)}, 
one may recover the morphism $g\colon X\to{\proj}^2$ from the class $j(h)$ by Lemma \ref{proj mor}, 
the points $p_i=g(E_i)$ from the classes $j(e_i)$, 
and the sextic $C_1+C_2$ from $g$ as the branch locus. 
Also one has the ample class $j(h+v)$ on $X$ defined in terms of $j$. 
Hence one may proceed as Section \ref{ssec: (r,r,1)} to see the following.

\begin{theorem}\label{period map (10,8,0)}
The period map $\tilde{p}\colon U\to\widetilde{\mathcal{M}}_{(10,8,0)}$ is a morphism of varieties 
and descends to an open immersion $U/{\PGL}\to\widetilde{\mathcal{M}}_{(10,8,0)}$. 
\end{theorem}

\begin{corollary}\label{unirat (10,<10,0)}
If $a\leq8$, then $\widetilde{\mathcal{M}}_{(10,a,0)}$ is unirational. 
\end{corollary}


\subsection{$\widetilde{\mathcal{M}}_{(10,8,1)}$ and bidegree $(3, 2)$ curves}\label{ssec: (10,8,1)}

Let $Q={\proj}^1\times{\proj}^1$ be a smooth quadric in ${\proj}^3$ and 
let $G ={\rm PGL}_2\times{\rm PGL}_2$. 
Let $U\subset|\mathcal{O}_{Q}(4, 4)|\times Q^{8}$ be the variety of 
pointed bidegree $(4, 4)$ curves $(C+D, \mathbf{p})=(C+D, p_1,\cdots, p_8)$ such that 
$(i)$ $C$ is smooth of bidegree $(3, 2)$, 
$(ii)$ $D$ is smooth of bidegree $(1, 2)$ and transverse to $C$, and 
$(iii)$ $C\cap D = \{ p_1,\cdots, p_8 \}$. 
The space $U$ is an $\frak{S}_8$-cover of an open set of $|\mathcal{O}_{Q}(3, 2)| \times |\mathcal{O}_{Q}(1, 2)|$. 
By Proposition \ref{stability of (4, 4) curve} and Lemma \ref{stab & fin mor}, 
we have a geometric quotient $U/G$ as a $10$-dimensional variety.  


\begin{lemma}\label{unirat parameter space (10,8,1)}
The variety $U$ is rational. 
\end{lemma}

\begin{proof}
Let $V$ be the linear system $|\mathcal{O}_{Q}(1, 2)|$ and 
$\mathcal{X} \subset V \times Q$ be the universal curve over $V$. 
The projection $\pi_1\colon\mathcal{X}\to V$ is birationally equivalent to the natural projection ${\proj}^{1}\times V \to V$ 
for bidegree $(0, 1)$ curves on $Q$ give sections of $\pi_1$. 
This implies that the fiber product 
$\mathcal{Y}=\mathcal{X} \mathop{\times}_{V}\mathcal{X} \cdots \mathop{\times}_{V}\mathcal{X}$ 
($8$ times) is rational. 
We have a morphism $\pi_2\colon U\to \mathcal{Y}$ defined by $(C+D, \mathbf{p}) \mapsto (D, \mathbf{p})$.  
Then $\pi_2$ is dominant. 
Indeed, for every smooth $D \in V$ 
the restriction map $|\mathcal{O}_{Q}(3, 2)| \dashrightarrow |\mathcal{O}_{D}(8)|$ is dominant 
by the vanishing of $H^1(\mathcal{O}_{Q}(2, 0))$. 
Since a general $\pi_2$-fiber is an open set of a linear subspace of $|\mathcal{O}_{Q}(3, 2)|$, 
this proves the rationality of $U$. 
\end{proof}

For a curve with labeling $(C+D, \mathbf{p})\in U$, 
let $(X, \iota)$ be the 2-elementary $K3$ surface associated to the DPN pair $(Q, C+D)$ 
and $g\colon X\to Q$ be the natural projection branched over $C+D$. 
The fixed curve $X^{\iota}$ is decomposed as 
$X^{\iota}=F_1+F_2$ such that $g(F_1)=C$ and $g(F_2)=D$. 
In this case, a reference lattice $L_+$ should be defined as follows. 
Let $M$ be the lattice $U(2)\oplus\langle-2\rangle^8 = \langle u, v, e_1,\cdots, e_8\rangle$ 
where $\{ u, v\}$ is the basis of $U(2)$ defined in $(\ref{U(2)})$ and 
$\{ e_1,\cdots, e_8\}$ is a natural basis of $\langle-2\rangle^8$. 
Let $f_1, f_2 \in M^{\vee}$ be the vectors defined by 
$2f_1=3u+2v-\sum_{i=1}^{8}e_i$ and $2f_2=u+2v-\sum_{i=1}^{8}e_i$. 
The overlattice $L_+\!=\!\langle M, f_1, f_2 \rangle$ 
is even and 2-elementary of main invariant $(10, 8, 1)$. 
Then, by Example \ref{ex:1}, we have a natural isometry $j\colon L_+\to L_+(X, \iota)$ by sending 
$u\mapsto[g^{\ast}\mathcal{O}_Q(1, 0)]$, $v\mapsto[g^{\ast}\mathcal{O}_Q(0, 1)]$, 
$e_i\mapsto[g^{-1}(p_i)]$, and $f_j\mapsto[F_j]$. 
In this way we associate to $(C+D, \mathbf{p})$ 
a lattice-marked 2-elementary $K3$ surface $((X, \iota), j)$, 
and hence a point $\tilde{p}(C+D, \mathbf{p})$ in $\widetilde{\mathcal{M}}_{(10,8,1)}$.

As in Section \ref{ssec: (2,2,0)}, 
the morphism $g\colon X\to Q$, the curve $C+D$, and the two rulings of $Q$ are recovered from $j$. 
The points $p_i$ are recovered from the classes $j(e_i)$. 
Therefore we have

\begin{theorem}\label{open immersion (10,8,1)}
The period map $\tilde{p}\colon U\to\widetilde{\mathcal{M}}_{(10,8,1)}$ is a morphism of varieties 
and descends to an open immersion $U/G \to \widetilde{\mathcal{M}}_{(10,8,1)}$. 
\end{theorem}

\begin{corollary}\label{unirat (10,<10,1)}
If $a\leq8$, then $\widetilde{\mathcal{M}}_{(10,a,1)}$ is unirational. 
\end{corollary}


\subsection{The unirationality of $\mathcal{M}_{(10,10,1)}$}\label{ssec: (10,10,1)}

By Theorem \ref{period map r<12}, general members of $\mathcal{M}_{(10,10,1)}$ 
are obtained from \textit{Halphen curves}, irreducible nine-nodal sextics. 
However, since the nodal map $\widetilde{V}_9 \to ({\proj}^{2})^9$ 
for Halphen curves is not dominant (see \cite{Coo} p.389--p.391), 
our proof of Proposition \ref{rationality of V_<10} does not apply to $\widetilde{V}_9$. 
Here we instead prove the unirationality of $\mathcal{M}_{(10,10,1)}$ 
using the description as a modular variety.

\begin{theorem}\label{unirat (10,10,1)}
The moduli space $\mathcal{M}_{(10, 10, 1)}$ is unirational. 
\end{theorem}

\begin{proof}
Recall that $\mathcal{M}_{(10,10,1)}$ is an open set of the arithmetic quotient $\mathcal{F}({\rm O}(L_1)^+)$ 
for the lattice $L_1 = U \oplus \langle 2 \rangle \oplus \langle -2 \rangle \oplus E_8(2)$. 
By Proposition \ref{duality of 2-ele lattice} we have an isomorphism 
$\mathcal{F}({\rm O}(L_1)^+) \simeq \mathcal{F}({\rm O}(L_2)^+)$ 
for the odd lattice $L_2 = U(2)\oplus\langle1\rangle\oplus\langle-1\rangle\oplus E_8$. 
Let $L_3$ be the lattice $U(2)^2\oplus E_8$ and 
$\{ u, v \}$ be the basis of its second summand $U(2)$ as defined in $(\ref{U(2)})$. 
Then $L_2$ is isometric to the overlattice $\langle L_3, \frac{1}{2}(u+v) \rangle $ of $L_3$. 
Thus $\mathcal{F}(\widetilde{{\rm O}}(L_2)^+)$ is dominated by $\mathcal{F}(\widetilde{{\rm O}}(L_3)^+)$ 
by Proposition \ref{construct isogeny}. 
The variety $\mathcal{F}(\widetilde{{\rm O}}(L_3)^+) = \widetilde{\mathcal{M}}_{(10,4,0)}$ is unirational 
by Corollary \ref{unirat (10,<10,0)}. 
Hence $\mathcal{F}({\rm O}(L_{1})^{+})$ is unirational. 
\end{proof}


\begin{remark}
Alternatively, considering morphisms to ${\proj}^2$ of genus $1$ and degree $6$, 
one can prove that $V_9$ is unirational using e.g., 
the relative Poincar\'e bundle for a rational elliptic surface with a section. 
\end{remark}


\section{The case $r=11$}\label{sec: r=11}

In this section we prove that $\mathcal{M}_{(11,11,1)}$ is unirational (Section \ref{ssec: (11,11,1)}) 
and that the covers $\widetilde{\mathcal{M}}_{(11,a,\delta)}$ are unirational for $a\leq 9$ (Section \ref{ssec: (11,9,1)}).


\subsection{$\mathcal{M}_{(11,11,1)}$ and Coble curves}\label{ssec: (11,11,1)}

Let $\widetilde{V}_{10}$ be the variety defined in $(\ref{marked Severi})$. 
By Theorem \ref{period map r<12} we have an open immersion 
$\widetilde{V}_{10}/{\PGL} \to \widetilde{\mathcal{M}}_{(11, 11, 1)}$ 
and hence a dominant morphism $\mathcal{P}\colon\widetilde{V}_{10}/{\PGL}\to\mathcal{M}_{(11,11,1)}$. 
Clearly, $\mathcal{P}$ descends to a morphism $V_{10}/{\PGL}\to\mathcal{M}_{(11,11,1)}$. 
The Severi variety $V_{10}$ is dense in the variety of rational plane sextics (cf. \cite{Ha}). 
As the latter is dominated by the variety of morphisms ${\proj}^1 \to {\proj}^2$ of degree $6$, 
which is obviously rational, we have the following. 

\begin{theorem}\label{unirat Coble}
The moduli space $\mathcal{M}_{(11,11,1)}$ is unirational. 
\end{theorem}


\subsection{$\widetilde{\mathcal{M}}_{(11,9,1)}$ and degenerated cubic pairs}\label{ssec: (11,9,1)}

Let $U \subset |{\sextic}|\times({\proj}^2)^8$ be the variety of pointed sextics 
$(C_1+C_2, \mathbf{p})=(C_1+C_2, p_1,\cdots, p_8)$ such that 
$C_1$ is a smooth cubic, 
that $C_2$ is an irreducible cubic with a node and transverse to $C_1$, 
and that $p_1,\cdots, p_8$ are distinct points contained in $C_1 \cap C_2$. 
Letting $p_9$ be the remaining intersection point of $C_1$ and $C_2$, 
and $p_{10}$ be the node of $C_2$, 
we have the complete labeling $(p_1,\cdots, p_{10})$ of the nodes of $C_1+C_2$. 
As in Section \ref{ssec: (10,8,0)}, 
we have a geometric quotient $U/{\PGL}$ as a $9$-dimensional variety.

\begin{lemma}\label{unirat paramet (11,9,1)}
The variety $U$ is unirational.  
\end{lemma}

\begin{proof}
Let $V$ denote the variety of irreducible cubics with nodes 
and $\mathcal{C} \subset V \times {\proj}^2$ be the universal curve over $V$. 
Let $\mathcal{X} = \mathcal{C} \mathop{\times}_{V}\mathcal{C} \cdots \mathop{\times}_{V}\mathcal{C}$ ($8$ times). 
We have a morphism $\pi\colon U \to\mathcal{X}$ defined by $(C_1+C_2, \mathbf{p})\mapsto(C_2, \mathbf{p})$. 
A general $\pi$-fiber is an open set of a line in $|{\cubic}|$, 
namely the linear system $|\!-\!K_Y|$ for the blow-up $Y$ of ${\proj}^2$ at $\{ p_i \} _{i=1}^{8}$. 
Therefore $U$ is birational to $\mathcal{X}\times {\proj}^1$. 
Take a nodal cubic $[C] \in V$. 
Since ${\PGL}\cdot [C] = V$, 
we have ${\PGL}\cdot (C)^8 = \mathcal{X}$ and hence $\mathcal{X}$ is unirational. 
\end{proof}

For a pointed sextic $(C_1+C_2, \mathbf{p})\in U$, 
the 2-elementary $K3$ surface $(X, \iota)$ associated to $C_1+C_2$ has main invariant $(11, 9, 1)$. 
As before, the above labeling of the nodes induces a natural isometry 
$j\colon L_+\to L_+(X, \iota)$ from a reference lattice $L_+$, 
and this defines a morphism $\tilde{p}\colon U\to\widetilde{\mathcal{M}}_{(11,9,1)}$. 
Then we see the following.

\begin{theorem}\label{open immersion (11,7,1)}
The period map $\tilde{p}$ descends to an open immersion $U/{\PGL}\to\widetilde{\mathcal{M}}_{(11,9,1)}$. 
\end{theorem}

\begin{corollary}\label{unirat (11,a<9,1)}
For $a\leq9$ the covers $\widetilde{\mathcal{M}}_{(11, a, \delta)}$ are unirational. 
\end{corollary}


\section{The case $r=12$}\label{sec: r=12}

In this section we study the case $r=12$. 
In Section \ref{ssec: (12,10,1)} we construct a birational map from the configuration space of 
eight general points in ${\proj}^2$ to a certain cover of $\mathcal{M}_{(12,10,1)}$, 
which in particular implies that $\mathcal{M}_{(12,10,1)}$ is unirational. 
In Section \ref{ssec: (12,8,1)} we prove that 
the covers $\widetilde{\mathcal{M}}_{(12,a,\delta)}$ for $a\leq8$ are unirational.


\subsection{$\mathcal{M}_{(12,10,1)}$ and eight general points in ${\proj}^2$}\label{ssec: (12,10,1)}

We begin by preparing lattices and an arithmetic group. 
Let $M_{12}=\langle h, e_1,\cdots, e_{11} \rangle$ be the lattice defined in $(\ref{Cremona lattice})$. 
Let $f_1, f_2 \in M_{12}^{\vee}$ be the vectors defined by 
$2f_i = 3h-2e_i-\sum_{j=3}^{11}e_j$,  $i=1, 2$. 
Then the overlattice $L_+ = \langle M_{12}, f_1, f_2 \rangle$   
is even and 2-elementary of main invariant $(12, 10, 1)$. 
We fix a primitive embedding $L_+ \subset \Lambda_{K3}$, which exists by Table \ref{Nikulin table}, 
and set $L_- = (L_+)^{\perp} \cap \Lambda_{K3}$. 
The lattice $L_-$ is isometric to $\langle2\rangle^2\oplus\langle-2\rangle^8$. 
We let the symmetric group $\frak{S}_3$ act on the set $\{ e_9, e_{10}, e_{11}\}$ by permutation, 
and on the set $\{ h, e_1,\cdots, e_8\}$ trivially. 
This defines an action $i\colon\frak{S}_3\to{\rm O}(L_+)$ of $\frak{S}_3$ on the lattice $L_+$. 
Let $r_{\pm}\colon{\rm O}(L_{\pm})\to{\rm O}(D_{L_{\pm}})$ be the natural homomorphisms 
and $\lambda\colon{\rm O}(D_{L_+})\simeq{\rm O}(D_{L_-})$ be 
the isomorphism induced by the relation $L_-=(L_+)^{\perp}$. 
Then we define a subgroup of ${\rm O}(L_-)$ by 
$\Gamma = r_-^{-1}(\lambda \circ r_+(i(\frak{S}_3)))$. 
By Proposition \ref{glue} an isometry $\gamma$ of $L_-$ is contained in $\Gamma$ if and only if 
there exists a $\sigma\in\frak{S}_3$ such that $i(\sigma)\oplus\gamma$ extends to an isometry of $\Lambda_{K3}$. 
We have $\widetilde{{\rm O}}(L_-)\subset \Gamma$ with 
$\Gamma^+/\widetilde{{\rm O}}(L_-)^+ \simeq \frak{S}_3$. 
Hence the modular variety $\mathcal{F}_{L_-}(\Gamma^{+})$ is a quotient of 
$\widetilde{\mathcal{M}}_{(12,10,1)}$ by $\frak{S}_3$.  
The moduli space $\mathcal{M}_{(12,10,1)}$ is dominated by $\mathcal{F}_{L_-}(\Gamma^+)$.

\begin{figure}[h]
\centerline{\includegraphics[width=6.5cm]{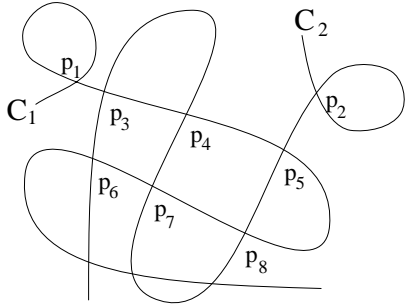}} 
\caption{Sextic curve for $(r, a, \delta)=(12, 10, 1)$} 
\label{r12a10}
\end{figure}

We shall define a parameter space. 
First we note that for seven general points $q_1,\cdots, q_7$ in ${\proj}^2$ 
there uniquely exists an irreducible nodal cubic $C$ passing $q_1,\cdots, q_7$ with ${\rm Sing}(C)=q_1$. 
This may be seen by an intersection calculation and a dimension counting. 
More constructively, 
the blow-up $Y$ of ${\proj}^2$ at $q_1,\cdots, q_7$ is a quadric del Pezzo surface 
which has the Geisser involution $\iota$. 
If $E\subset Y$ is the $(-1)$-curve over $q_1$, 
then the image of the curve $\iota(E)$ in ${\proj}^2$ is the desired cubic. 
Now let $U \subset ({\proj}^2)^8$ be the open set of eight distinct points 
$\mathbf{p}=(p_1, \cdots, p_8)$ such that 
there exist irreducible nodal cubics $C_1, C_2$ which pass $p_3,\dots, p_8$ with 
${\rm Sing}(C_i)=p_i$ and which are transverse to each other. 
The finite morphism $U\to|{\sextic}|$, $\mathbf{p}\mapsto C_1+C_2$, shows that 
we have a geometric quotient $U/{\PGL}$ as an $8$-dimensional variety,  
which is rational by Proposition \ref{existence of configuration space}.

For a $\mathbf{p}=(p_1, \cdots, p_8) \in U$ 
the associated sextic $C_1+C_2$ is endowed with the partial labeling $(p_1, \cdots, p_8)$ of its nodes. 
The remaining  three nodes $\mathcal{S}=C_1\cap C_2\backslash \{ p_i\}_{i=3}^{8}$ are not marked. 
We \textit{temporarily} choose a bijection $\mathcal{S}\simeq \{9, 10, 11\}$ 
and accordingly denote $\mathcal{S} = \{p_9, p_{10}, p_{11}\}$. 
Then let $(X, \iota)$ be the 2-elementary $K3$ surface associated to $C_1+C_2$. 
If $g\colon X\to{\proj}^2$ is the natural projection branched over $C_1+C_2$, 
we have an isometry $j\colon L_+\to L_+(X, \iota)$ defined by 
$h\mapsto[g^{\ast}{\HP}]$, 
$e_i\mapsto[g^{-1}(p_i)]$ for $i\leq11$, 
and $f_j\mapsto[F_j]$ where $F_j$ is the component of $X^{\iota}$ with $g(F_j)=C_j$. 
Then the period of $((X, \iota), j)$ is determined 
as a point in $\widetilde{\mathcal{M}}_{(12,10,1)}$. 
We consider the image of that point in $\mathcal{F}_{L_-}(\Gamma^+)$, 
and denote it by $\mathcal{P}(\mathbf{p})\in\mathcal{F}_{L_-}(\Gamma^+)$.

\begin{theorem}\label{period map (12,10,1)}
The map $\mathcal{P}\colon U\to\mathcal{F}_{L_-}(\Gamma^+)$ is well-defined. 
It is a morphism of varieties and induces an open immersion $U/{\PGL} \to \mathcal{F}_{L_-}(\Gamma^+)$.  
\end{theorem}

\begin{proof}
For the first assertion 
it suffices to show that $\mathcal{P}(\mathbf{p})$ is independent of 
the choice of a labeling $\mathcal{S} = \{p_9, p_{10}, p_{11}\}$. 
For another labeling $\mathcal{S} = \{p_9', p_{10}', p_{11}'\}$ 
we have $p_{\sigma(i)}=p_i'$ for a $\sigma\in\frak{S}_3$, $9\leq i\leq11$. 
Then the isometry $j'\colon L_+\to L_+(X, \iota)$ associated to $(p_9', p_{10}', p_{11}')$ 
is given by $j'=j\circ i(\sigma)$. 
If $\Phi, \Phi'\colon\Lambda_{K3}\to{\cohomology}$ are extensions of $j$ and $j'$ respectively, 
then $\Phi|_{L_-}$ is $\Gamma$-equivalent to $\Phi'|_{L_-}$. 

The map $\mathcal{P}$ is obviously ${\PGL}$-invariant. 
Conversely, suppose that $\mathcal{P}(\mathbf{p})=\mathcal{P}(\mathbf{p}')$ 
for two $\mathbf{p}, \mathbf{p}'\in U$.  
We choose labelings of the three nodes for $\mathbf{p}$ and $\mathbf{p}'$ respectively, 
and let $(X, j)$ and $(X', j')$ be the associated marked $K3$ surfaces. 
Then the equality $\mathcal{P}(\mathbf{p})=\mathcal{P}(\mathbf{p}')$ means that 
we have a Hodge isometry $\Phi\colon{\cohomology}\to H^2(X', {\Z})$ 
with $\Phi\circ j=j'\circ i(\sigma)$ for some $\sigma\in\frak{S}_3$. 
In particular, we have $\Phi(j(h))=j'(h)$, $\Phi(j(f_j))=j'(f_j)$, and $\Phi(j(e_i))=j'(e_i)$ for $i\leq8$. 
As before, 
we deduce that $\mathbf{p}$ and $\mathbf{p}'$ are ${\PGL}$-equivalent. 
This concludes the proof. 
\end{proof}

\begin{corollary}
The variety $\mathcal{F}_{L_-}(\Gamma^+)$ is rational. 
Hence $\mathcal{M}_{(12,10,1)}$ is unirational. 
\end{corollary}

\begin{remark}\label{mkd dP of deg 1}
The space $U/{\PGL}$ is birationally identified with the moduli of \textit{marked} del Pezzo surfaces of degree $1$. 
It would be interesting to study the rational action of the Weyl group on 
$\mathcal{F}_{L_-}(\Gamma^+)$ induced by the above immersion. 
Kond\=o \cite{Ko1.5} described the moduli of del Pezzo surfaces of degree $1$ as a ball quotient. 
\end{remark}


\subsection{The unirationality of $\widetilde{\mathcal{M}}_{(12,8,1)}$}\label{ssec: (12,8,1)}

Let $U \subset |{\cubic}|\times ({\proj}^2)^8$ be the locus of cubics with points 
$(C, \mathbf{p})=(C, p_1, \cdots, p_8)$ such that 
$(i)$ $p_1, \cdots, p_8$ are distinct, 
$(ii)$ $C$ is smooth and passes $\{ p_i \} _{i\ne6}$, 
$(iii)$ $p_1, \cdots, p_6$ lie on a smooth conic $Q$, 
$(iv)$ $p_6, p_7, p_8$ lie on a line $L$, and 
$(v)$ $C+Q+L$ has only nodes as singularities. 
The sextic $C+Q+L$ is uniquely determined by $(C, \mathbf{p})$. 
By setting 
$p_9 = L \cap C \backslash \{ p_7, p_8 \}$,  
$p_{10} = L \cap Q \backslash p_6$,  
and $p_{11} = Q \cap C \backslash \{ p_i \}_{i=1}^{5}$, 
we have a complete marking of the nodes of $C+Q+L$. 
For the proof of unirationality it is convenient to reduce sextics with labelings to such cubics with points, 
and consider the space $U$ of the latters. 

\begin{figure}[h]
\centerline{\includegraphics[width=6.5cm]{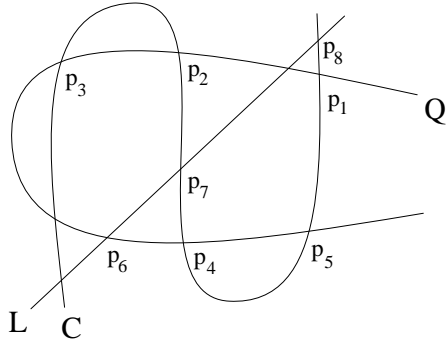}} 
\caption{Sextic curve for $(r, a, \delta)=(12, 8, 1)$} 
\label{r12a8}
\end{figure}

\begin{lemma}\label{unirat paramet (12,8,1)}
The variety $U$ is unirational. 
\end{lemma}

\begin{proof}
Let $V \subset ({\proj}^2)^6$ be the locus of six points $(p_1, \cdots, p_6)$ lying on some conic
and $W \subset ({\proj}^2)^3$ be the locus of three collinear points $(q_1, q_2, q_3)$. 
The fiber product $V \mathop{\times}_{{\proj}^2} W$ over ${\proj}^2= \{ p_6 \in {\proj}^2 \} =\{ q_1 \in {\proj}^2 \}$ 
is birational to the image of the projection $U\to({\proj}^2)^8$, $(C, \mathbf{p}) \mapsto \mathbf{p}$. 
As a general fiber of the projection $U\to V\mathop{\times}_{{\proj}^2}W$ is an open set of a plane in $|{\cubic}|$, 
it suffices to prove the unirationality of $V \mathop{\times}_{{\proj}^2} W$, 
which is easily reduced to that of $V$. 
Let $p_1, \cdots, p_4 \in {\proj}^2$ be four general points and  
$S$ be the blow-up of ${\proj}^2$ at $\{ p_i \}_{i=1}^{4}$.  
The conic pencil determined by $\{ p_i \}_{i=1}^{4}$ defines a morphism $S \to {\proj}^1$. 
We have a birational map ${\PGL} \times (S\mathop{\times}_{{\proj}^1} S) \dashrightarrow V$. 
Then the existence of sections of $S \to {\proj}^1$ implies the rationality of $S\mathop{\times}_{{\proj}^1} S$. 
\end{proof}

For a $(C, \mathbf{p})\in U$, 
the 2-elementary $K3$ surface $(X, \iota)$ associated to 
the sextic $C+Q+L$ has main invariant $(12, 8, 1)$. 
As before, our labeling for $C+Q+L$ 
will induce an isometry $j\colon L_+ \to L_+(X, \iota)$ 
from an appropriate reference lattice $L_+$. 
This defines a morphism $\tilde{p}\colon U\to\widetilde{\mathcal{M}}_{(12,8,1)}$, 
and we have the following.

\begin{theorem}\label{open immersion}
The period map $\tilde{p}$ descends to an open immersion 
$U/{\PGL}\to\widetilde{\mathcal{M}}_{(12,8,1)}$ from a geometric quotient $U/{\PGL}$. 
\end{theorem}

\begin{corollary}\label{unirat (12, a<10)}
For $a\leq8$ the covers $\widetilde{\mathcal{M}}_{(12,a,\delta)}$ are unirational. 
\end{corollary}


\section{The case $r=13$}\label{sec: r=13}

In this section we study the case $r=13$. 
In Section \ref{ssec: (13,9,1)} we 
construct a birational map from 
a configuration space of eight special points in ${\proj}^2$ 
to a certain cover of $\mathcal{M}_{(13,9,1)}$ 
in a similar way as Section \ref{ssec: (12,10,1)}. 
In Section \ref{ssec: (13,7,1)} we prove that 
the covers $\widetilde{\mathcal{M}}_{(13,a,\delta)}$ are unirational for $a\leq7$.


\subsection{$\mathcal{M}_{(13,9,1)}$ and eight special points in ${\proj}^2$}\label{ssec: (13,9,1)}

Let $M_{13}=\langle h, e_1, \cdots, e_{12} \rangle$ be the lattice defined in $(\ref{Cremona lattice})$. 
We define the vectors $f_1, f_2, f_3 \in M_{13}^{\vee}$ by 
$2f_3 = 3h-2e_1-\sum_{i=3}^{11}e_i$, 
$2(f_1+f_2) = 3h-2(e_2+e_{12})-\sum_{i=3}^{11}e_i$, 
and $2f_2 = 2h-(e_2+e_{12})-\sum_{i=5}^{10}e_i$. 
The overlattice $L_+=\langle M_{13}, f_1, f_2, f_3\rangle$   
is 2-elementary of main invariant $(13, 9, 1)$. 
We let $\frak{S}_2$ act on $L_+$ by the permutation on $\{ e_9, e_{10} \}$.  
We fix a primitive embedding $L_+ \subset \Lambda_{K3}$ and set 
$L_-=(L_+)^{\perp}\cap\Lambda_{K3}$. 
The lattice $L_-$ is isometric to $\langle2\rangle^2\oplus\langle-2\rangle^7$. 
Then let $\Gamma \subset{\rm O}(L_-)$ be the group  
$r_-^{-1}(\lambda \circ r_+(\frak{S}_2))$,  
where $r_{\pm}\colon{\rm O}(L_{\pm})\to{\rm O}(D_{L_{\pm}})$ and 
$\lambda\colon{\rm O}(D_{L_+})\to{\rm O}(D_{L_-})$ 
are defined as in Section \ref{ssec: (12,10,1)}. 
The arithmetic quotient $\mathcal{F}_{L_-}(\Gamma^+)$ is 
a quotient of $\widetilde{\mathcal{M}}_{(13,9,1)}$ by $\frak{S}_2$,   
and dominates $\mathcal{M}_{(13,9,1)}$. 

Let $V \subset ({\proj}^2)^8$ be the codimension $1$ locus of 
eight distinct points $\mathbf{p}=(p_1, \dots, p_8)$ such that 
$(i)$ there exists an irreducible nodal cubic $C$ passing $\{ p_i \}_{i\ne2}$ with ${\rm Sing}(C)=p_1$, 
$(ii)$ $p_2$ lies on the line $L=\overline{p_3p_4}$,  
$(iii)$ there exists a smooth conic $Q$ passing $\{ p_2 \} \cup \{ p_i \}_{i=5}^{8}$, and 
$(iv)$ the sextic $C+Q+L$ has only nodes as singularities. 
We shall denote $p_{11} = L\cap C \backslash \{ p_3, p_4 \}$ and $p_{12} = L \cap Q \backslash p_2$. 
In this way we obtain from $\mathbf{p}$ the sextic $C+Q+L$ and the partial labeling 
$(p_1,\cdots, p_8, p_{11}, p_{12})$ of its nodes. 
The remaining two nodes $\mathcal{S}=Q\cap C\backslash \{ p_i\}_{i=5}^{8}$ are not naturally marked. 
We have a geometric quotient $V/{\PGL}$ as a $7$-dimensional variety, 
which is rational by Proposition \ref{existence of configuration space}. 

\begin{figure}[h]
\centerline{\includegraphics[width=6.5cm]{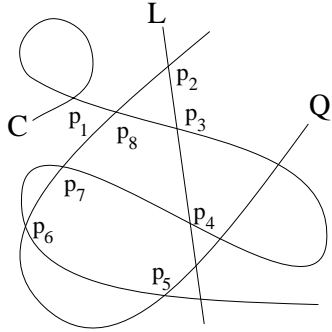}} 
\caption{Sextic curve for $(r, a, \delta)=(13, 9, 1)$} 
\label{r13a9}
\end{figure}

For a $\mathbf{p}\in V$, 
let $(X, \iota)$ be the 2-elementary $K3$ surface associated to the sextic $C+Q+L$. 
A temporary choice of a labeling $\mathcal{S}=\{ p_9, p_{10}\}$ induces a natural isometry $j\colon L_+\to L_+(X, \iota)$,  
which defines a point in $\widetilde{\mathcal{M}}_{(13,9,1)}$ as the period of $((X, \iota), j)$. 
Considering the image in $\mathcal{F}_{L_-}(\Gamma^+)$ of the period of $((X, \iota), j)$, 
we obtain a well-defined morphism $\mathcal{P}\colon V \to \mathcal{F}_{L_-}(\Gamma^+)$ 
as in Section \ref{ssec: (12,10,1)}. 
Then we have the following. 

\begin{theorem}\label{open immersion}
The period map $\mathcal{P}$ descends to an open immersion $V/{\PGL}\to\mathcal{F}_{L_-}(\Gamma^+)$. 
In particular, $\mathcal{F}_{L_-}(\Gamma^{+})$ is rational and $\mathcal{M}_{(13,9,1)}$ is unirational. 
\end{theorem}


\subsection{$\widetilde{\mathcal{M}}_{(13,7,1)}$ and pointed cubics}\label{ssec: (13,7,1)}

Let $U \subset |{\cubic}| \times ({\proj}^2)^6$ be the space of 
pointed cubics $(C, \mathbf{p}) = (C, p_{1+}, p_{1-}, \cdots, p_{3-})$ such that 
$(i)$ $C$ is smooth, 
$(ii)$ $p_{1+}, \cdots, p_{3-}$ are distinct points on $C$, and 
$(iii)$ if we denote $L_i=\overline{p_{i+}p_{i-}}$, the sextic $C+\sum_iL_i$ has only nodes as singularities. 
The variety $U$ is rational, 
for the natural projection $U\to({\proj}^2)^6$ 
is birational to the projectivization of a vector bundle on an open set. 
For a pointed cubic $(C, \mathbf{p})\in U$ we set 
$p_i=L_i\cap C\backslash \{ p_{i+}, p_{i-} \}$ and $q_i= L_j\cap L_k$ where $\{ i, j, k\}=\{1, 2, 3\}$. 
Thus we associate to $(C, \mathbf{p})$ 
the nodal sextic $C+\sum_iL_i$ with the labeling $(p_{\mu}, q_i)_{\mu, i}$ of its nodes. 
As before, 
from these we will obtain a lattice-marked 2-elementary $K3$ surface $((X, \iota), j)$ of type $(13, 7, 1)$. 
This defines a morphism $\tilde{p}\colon U\to\widetilde{\mathcal{M}}_{(13,7,1)}$, 
and we have the following. 

\begin{theorem}\label{open immersion (13,7,1)}
The period map $\tilde{p}$ descends to an open immersion $U/{\PGL}\to\widetilde{\mathcal{M}}_{(13,7,1)}$ 
from a geometric quotient $U/{\PGL}$. 
\end{theorem}

\begin{corollary}\label{unirat (13, a<9)}
The covers $\widetilde{\mathcal{M}}_{(13,a,\delta)}$ for $a\leq7$ are unirational. 
\end{corollary}


\section{The case $r\geq 14$}\label{sec: r>13} 

Let $U_d, V_d \subset ({\proj}^{2})^d$ be the loci defined in Section \ref{subsec: GIT}. 
By Proposition \ref{existence of configuration space}, when $d\geq5$,  
we have geometric quotients $U_d/{\PGL}$ and $V_d/{\PGL}$ 
as rational varieties of dimension $2d-8$ and $2d-9$ respectively. 
In this section we prove the following. 

\begin{theorem}\label{main  r>13}
One has birational period maps 
$U_d/{\PGL}\dashrightarrow \widetilde{\mathcal{M}}_{(28-2d,2d-6,\delta)}$ and  
$V_d/{\PGL}\dashrightarrow \widetilde{\mathcal{M}}_{(29-2d,2d-7,1)}$ 
for $5\leq d\leq7$ . 
\end{theorem}

By Proposition \ref{isogeny bet discri cover} and Figure \ref{Nikulin table} we have the following corollary, 
which completes the proof of Theorem \ref{main}. 

\begin{corollary}
The covers ${\cover}$ are unirational for $r\geq14$. 
\end{corollary}

Our constructions of the period maps are similar to those for eight points 
(Sections \ref{ssec: (12,10,1)} and \ref{ssec: (13,9,1)}): 
we draw a sextic from a given point set, 
label its singularities in a natural way, 
and then associate a lattice-marked 2-elementary $K3$ surface. 
Unlike the eight point cases, 
our labelings for $d\leq7$ leave no ambiguity, 
and so we obtain points in $\widetilde{\mathcal{M}}_{(r,22-r,\delta)}$. 
Actually, these period maps may be derived from the ones for eight points by degeneration: 
as we specialize a configuration of points, the resulting sextic gets more degenerate, 
and the period goes to a Heegner divisor. 

Theorem \ref{main  r>13} for $U_6$ was first found by Matsumoto-Sasaki-Yoshida \cite{M-S-Y}. 
Considering degeneration, 
they essentially obtained the assertion also for $V_6$, $U_5$, $V_5$ with $\delta=1$. 
The novelty of Theorem \ref{main  r>13} is the constructions for $d=7$. 
But even for $d\leq6$, our period maps differ from the ones in \cite{M-S-Y}. 
Specifically, from a given point set we draw lines on the same plane, 
while in \cite{M-S-Y} the point set is regarded as a set of lines on the dual plane. 
Our argument as explained in Section \ref{ssec:discri cover} 
makes it easier to derive the monodromy groups, 
which were found by direct calculations in \cite{M-S-Y}. 



\subsection{$\widetilde{\mathcal{M}}_{(14,8,1)}$ and seven general points in ${\proj}^2$}\label{ssec: (14,8,1)}

Let $U\subset({\proj}^2)^7$ be the open set of seven distinct points 
$\mathbf{p}=(p_1, \cdots, p_7)$ such that 
$(i)$ there exists an irreducible nodal cubic $C$ passing $p_1, \cdots, p_7$ with ${\rm Sing}(C)=p_7$ and 
$(ii)$ if we denote $L_i=\overline{p_{i}p_{i+3}}$ for $i\leq3$, 
the sextic $C+\sum_iL_i$ has only nodes as singularities. 
We put $q_i=L_i \cap C \backslash \{ p_{i}, p_{i+3} \}$ and $q_{ij}=L_i \cap L_j$. 
We thus obtain from $\mathbf{p}$ the nodal sextic $C+\sum_iL_i$ and 
the complete labeling $(p_i, q_{\mu})_{i, \mu}$ of its nodes. 
The components of $C+\sum_iL_i$ are also labelled obviously. 
Taking the right resolution of $C+\sum_iL_i$ and using these labelings, 
we obtain a lattice-marked 2-elementary $K3$ surface $((X, \iota), j)$ of type $(14, 8, 1)$ as before. 
This defines a morphism $\tilde{p}\colon U\to\widetilde{\mathcal{M}}_{(14,8,1)}$,  
and we will see the following.

\begin{theorem}\label{period map (14,8,1)}
The period map $\tilde{p}$ descends to an open immersion $U/{\PGL}\to\widetilde{\mathcal{M}}_{(14,8,1)}$ 
from a geometric quotient $U/{\PGL}$. 
\end{theorem}

\begin{figure}[h]
\centerline{\includegraphics[width=6.5cm]{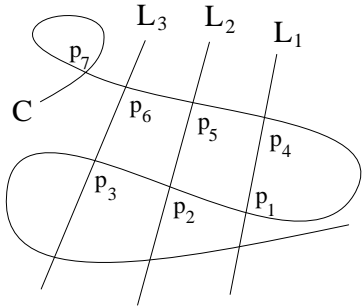}} 
\caption{Sextic curve for $(r, a, \delta)=(14, 8, 1)$} 
\label{r14}
\end{figure}

In the next section 
we degenerate the points $p_5, p_6, p_7$ to collinear position. 
This forces the cubic $C$ to degenerate to the union of a conic and a line.


\subsection{$\widetilde{\mathcal{M}}_{(15,7,1)}$ and seven special points in ${\proj}^2$}\label{subsec: (15, 7, 1)}

Let $V\subset({\proj}^2)^7$ be the codimension $1$ locus of 
seven distinct points $\mathbf{p} =(p_1, \cdots, p_7)$ such that 
$(i)$ $p_5, p_6, p_7$ lie on a line $L_0$, 
$(ii)$ $p_1,\cdots, p_4, p_7$ lie on a smooth conic $Q$, and 
$(iii)$ if we put $L_i=\overline{p_{i}p_{i+3}}$ for $1\leq i\leq3$, 
the sextic $Q+ \sum_{i=0}^{3}L_i$ has only nodes as singularities. 
We set $q_0=L_0\cap Q\backslash p_7$, 
$q_i=L_i\cap Q\backslash p_i$ for $i=2, 3$, 
and $q_{ij}=L_i\cap L_j$ when $q_{ij}\ne p_k$ for some $k$. 
In this way we obtain from $\mathbf{p}$ 
the sextic $Q+ \sum_iL_i$, 
the labeling $(p_i, q_{\mu})_{i, \mu}$ of its nodes, 
and also the obvious labeling of its components. 
As before, from these 
we obtain a lattice-marked 2-elementary $K3$ surface of type $(15, 7, 1)$. 
This defines a morphism $\tilde{p}\colon V\to\widetilde{\mathcal{M}}_{(15,7,1)}$, 
and we have the following. 

\begin{theorem}\label{period map (15,7,1)}
The period map $\tilde{p}$ 
descends to an open immersion $V/{\PGL}\to\widetilde{\mathcal{M}}_{(15,7,1)}$ 
from a geometric quotient $V/{\PGL}$. 
\end{theorem}

In the next section we degenerate $p_7$ on $\overline{p_1p_2}$. 
Then $p_7$ is determined as $\overline{p_1p_2}\cap\overline{p_5p_6}$, 
so that the parameters are reduced to six points. (We make renumbering).  

\begin{figure}[h]
\centerline{\includegraphics[width=6.5cm]{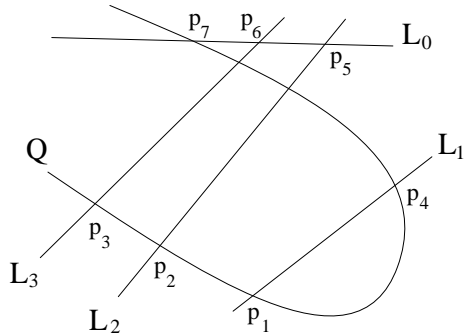}} 
\caption{Sextic curve for $(r, a, \delta)=(15, 7, 1)$} 
\label{r15}
\end{figure}


\subsection{$\widetilde{\mathcal{M}}_{(16,6,1)}$ and six general points in ${\proj}^2$}\label{ssec: (16,6,1)}

Let $U\subset({\proj}^2)^6$ be the open set of six distinct points 
$\mathbf{p}=(p_1,\cdots, p_6)$ such that if we draw six lines by  
$L_1=\overline{p_1p_2},\cdots, L_5=\overline{p_5p_6}$, and $L_6=\overline{p_6p_1}$, 
then the sextic $\sum_iL_i$ has only nodes as singularities. 
Since the nodes of $\sum_iL_i$ are the intersections of the lines $L_i$, 
the labeling $(L_1,\cdots, L_6)$ of the lines induces that of the nodes, 
e.g., by setting $p_{ij}=L_i\cap L_j$. 
Hence from $\mathbf{p}$ we obtain the sextic $\sum_iL_i$ 
with a labeling of its nodes and components. 
This defines a lattice-marked 2-elementary $K3$ surface of type $(16, 6, 1)$. 
Thus we obtain a morphism $\tilde{p}\colon U\to\widetilde{\mathcal{M}}_{(16,6,1)}$, 
and see the following. 

\begin{theorem}\label{period map for (16,6,1)}
The period map $\tilde{p}$ descends to an open immersion 
$\widetilde{\mathcal{P}}\colon U/{\PGL}\to\widetilde{\mathcal{M}}_{(16,6,1)}$ 
from a geometric quotient $U/{\PGL}$. 
\end{theorem}

\begin{remark}
If we identify ${\proj}^2\simeq|{\HP}|$, 
the assignment $\mathbf{p}\mapsto(L_1,\cdots, L_6)$ induces 
a Cremona transformation $w$ of $U/{\PGL}$. 
The period map of \cite{M-S-Y} is written as $\widetilde{\mathcal{P}} \circ w^{-1}$. 
One sees that $w^2$ is the cyclic permutation $(654321)$ on $U/{\PGL}$. 
\end{remark}
 
\begin{figure}[h]
\centerline{\includegraphics[width=6.4cm]{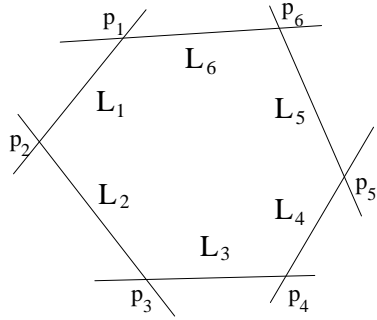}} 
\caption{Sextic curve for $(r, a, \delta)=(16, 6, 1)$} 
\label{r16}
\end{figure}


\subsection{$\widetilde{\mathcal{M}}_{(17,5,1)}$ and six special points in ${\proj}^2$}\label{ssec: (17,5,1)}

Let $V\subset ({\proj}^2)^6$ be the codimension $1$ locus of six distinct points 
$\mathbf{p}=(p_1,\cdots, p_6)$ such that 
$(i)$ $p_3, p_4, p_6$ are collinear, and  
$(ii)$ if we draw lines by $L_1=\overline{p_1p_2},\cdots, L_5=\overline{p_5p_6}$, and $L_6=\overline{p_6p_1}$, 
then any singularity of the sextic $\sum_iL_i$ other than $p_6$ is a node. 
The point $p_6$ is an ordinary triple point of $\sum_iL_i$. 
As in Section \ref{ssec: (16,6,1)}, 
we obtain a labeling of the nodes of $\sum_iL_i$ from the obvious one of the lines $L_i$. 
Denoting by $q_i$ the infinitely near point of $p_6$ given by $L_i$ for $i=3, 5, 6$, 
we also obtain a labeling of the branches of $\sum_iL_i$ at $p_6$. 
The 2-elementary $K3$ surface $(X, \iota)$ associated to the sextic $\sum_iL_i$ 
has main invariant $(17, 5, 1)$. 
Here we encounter a triple point, 
but we can proceed as before referring to Example \ref{ex:2}: 
if $g\colon X\to{\proj}^2$ is the natural projection branched over $\sum_iL_i$, 
the curve $g^{-1}(p_6)$ over $p_6$ consists of four labelled $(-2)$-curves, 
namely the $(-2)$-curves over $q_i$ and a component of $X^{\iota}$. 
Together with the above labeling for the nodes and the lines, 
this induces an isometry $j\colon L_+\to L_+(X, \iota)$ from a reference lattice $L_+$. 
Thus we obtain a morphism $\tilde{p}\colon V\to\widetilde{\mathcal{M}}_{(17,5,1)}$, 
and see the following.

\begin{theorem}\label{period map (17,5,1)}
The period map $\tilde{p}$ descends to an open immersion $V/{\PGL}\to\widetilde{\mathcal{M}}_{(17,5,1)}$ 
from a geometric quotient $V/{\PGL}$. 
\end{theorem}

\begin{figure}[h]
\centerline{\includegraphics[width=6.5cm]{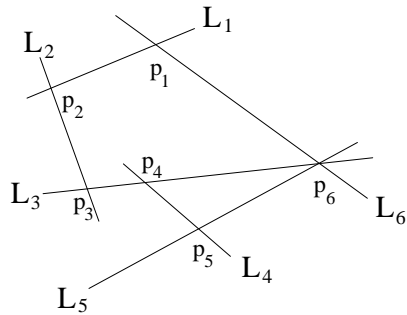}} 
\caption{Sextic curve for $(r, a, \delta)=(17, 5, 1)$} 
\label{r17}
\end{figure}

Degenerating $p_2, p_4, p_5$ to collinear position produces 
a period map for $\widetilde{\mathcal{M}}_{(18,4,0)}$ (Section \ref{ssec: (18,4,0)}), 
while degenerating $p_4$ to $p_3$ produces 
that for $\widetilde{\mathcal{M}}_{(18,4,1)}$ (Section \ref{ssec: (18,4,1)}).


\subsection{$\widetilde{\mathcal{M}}_{(18,4,0)}$ and five general points in ${\proj}^2$}\label{ssec: (18,4,0)}

Let $U\subset ({\proj}^2)^{5}$ be the open set of 
five distinct points $\mathbf{p}=(p_1,\cdots, p_5)$ such that 
no three of $p_1,\cdots, p_5$ other than $\{ p_1, p_2, p_3\}$ are collinear. 
For a $\mathbf{p}\in U$ we draw six lines by 
$L_i=\overline{p_ip_4}$ for $1\leq i\leq3$ and $L_i=\overline{p_{i-3}p_5}$ for $4\leq i\leq6$. 
Then the sextic $\sum_{i=1}^6L_i$ has ordinary triple points at $p_4$ and $p_5$, 
nodes at $L_i\cap L_j$ for $i\leq3$ and $j\geq4$, and no other singularity. 
The obvious labeling of the lines $L_i$ induces that of 
the nodes and the branches at the triple points of $\sum_iL_i$. 
The 2-elementary $K3$ surface $(X, \iota)$ associated to $\sum_iL_i$ 
has invariant $(r, a)=(18, 4)$. 
We have to identify its parity $\delta$. 
Let $(Y, B, \pi)$ be the right resolution of $\sum_iL_i$. 
We have the decomposition $B=\sum_{i=0}^7B_i$ such that 
$\pi(B_i)=L_i$ for $1\leq i\leq6$ and $\pi(B_0)=p_5$, $\pi(B_7)=p_4$. 
One checks that the divisor $(\sum_{i=0}^3B_i)-(\sum_{i=4}^7B_i)$ is in $4NS_Y$. 
Hence $(X, \iota)$ has parity $\delta=0$. 
Using our labeling for $\sum_iL_i$, 
we will obtain a morphism $\tilde{p}\colon U\to\widetilde{\mathcal{M}}_{(18,4,0)}$. 
Then we have the following.

\begin{theorem}\label{period map (18,4 0)}
The period map $\tilde{p}$ descends to an open immersion 
$U/{\PGL}\to\widetilde{\mathcal{M}}_{(18,4,0)}$ from a geometric quotient $U/{\PGL}$. 
\end{theorem}

\begin{figure}[h]
\centerline{\includegraphics[width=6.5cm]{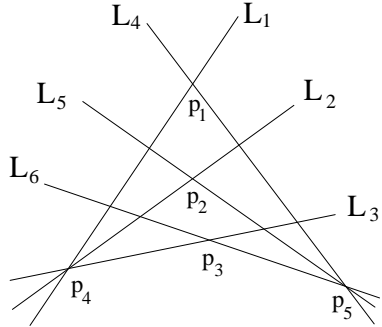}} 
\caption{Sextic curve for $(r, a, \delta)=(18, 4, 0)$} 
\label{r18_0}
\end{figure}


\subsection{$\widetilde{\mathcal{M}}_{(18,4,1)}$ and five general points in ${\proj}^2$}\label{ssec: (18,4,1)}

Let $U_5 \subset({\proj}^2)^5$ be the open set defined in  Section \ref{subsec: GIT}. 
To a point $\mathbf{p}=(p_1,\cdots, p_5)$ in $U_5$ 
we associate six lines by  
$L_1=\overline{p_2p_3}$,  
$L_i=\overline{p_1p_{i+2}}$ for $i=2, 3$,  
$L_i=\overline{p_ip_{i-2}}$ for $i=4, 5$, and   
$L_6=\overline{p_4p_5}$. 
The sextic $\sum_iL_i$ has ordinary triple points at $p_4$ and $p_5$. 
Any other singularity of $\sum_iL_i$ is a node. 
The 2-elementary $K3$ surface $(X, \iota)$ associated to $\sum_iL_i$ 
has invariant $(r, a)=(18, 4)$. 
In order to determine its parity $\delta$, 
let $g\colon X\to{\proj}^2$ be the natural projection branched over $\sum_iL_i$, 
and let $E_{ij}$ be the $(-2)$-curves $g^{-1}(L_i\cap L_j)$ for $i, j\leq3$. 
Then the ${\Q}$-divisor $D=\frac{1}{2}(E_{12}+E_{23}+E_{31})$ is in $L_+(X, \iota)^{\vee}$ 
by Proposition \ref{generate L_+}. 
Since $(D. D)=-\frac{3}{2}$, $(X, \iota)$ has parity $\delta=1$. 
Using the obvious labeling of the lines $L_i$,  
we obtain a morphism $\tilde{p}\colon U_5\to\widetilde{\mathcal{M}}_{(18,4,1)}$ as before. 
Then we see the following. 

\begin{theorem}\label{period map (18,4,1)}
The period map $\tilde{p}$ 
descends to an open immersion $U_5/{\PGL} \to \widetilde{\mathcal{M}}_{(18,4,1)}$. 
\end{theorem}

\begin{figure}[h]
\centerline{\includegraphics[width=6.4cm]{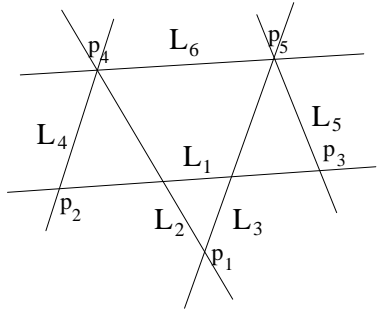}} 
\caption{Sextic curve for $(r, a, \delta)=(18, 4, 1)$} 
\label{r18_1}
\end{figure}


\subsection{$\widetilde{\mathcal{M}}_{(19,3,1)}$ and five special points in ${\proj}^2$}\label{subsec: (19, 3, 1)}

Let $V_5\subset({\proj}^2)^5$ be the codimension $1$ locus defined in Section \ref{subsec: GIT}. 
Given a point $\mathbf{p}=(p_1, \cdots, p_5)$ in $V_5$, 
for which $p_1, p_2, p_3$ are collinear, 
we define six lines in the same way as Section \ref{ssec: (18,4,1)}:  
$L_1=\overline{p_2p_3}$,  
$L_i=\overline{p_1p_{i+2}}$ for $i=2, 3$,  
$L_i=\overline{p_ip_{i-2}}$ for $i=4, 5$, and   
$L_6=\overline{p_4p_5}$. 
Then the points $p_1, p_4, p_5$ are ordinary triple points of the sextic $\sum_iL_i$, 
and any other singularity of $\sum_iL_i$ is a node. 
As before, 
by taking the right resolution of the sextic $\sum_iL_i$ and 
using the labeling $(L_1,\cdots, L_6)$ of the lines, 
we obtain a lattice-marked 2-elementary $K3$ surface of type $(19, 3, 1)$. 
This defines a morphism $\tilde{p}\colon V_5\to\widetilde{\mathcal{M}}_{(19,3,1)}$. 
Then we have the following. 

\begin{theorem}\label{period map (19,3,1)}
The period map $\tilde{p}$ 
descends to an open immersion $V_5/{\PGL} \to \widetilde{\mathcal{M}}_{(19,3,1)}$. 
\end{theorem}

\begin{figure}[h]
\centerline{\includegraphics[width=6.5cm]{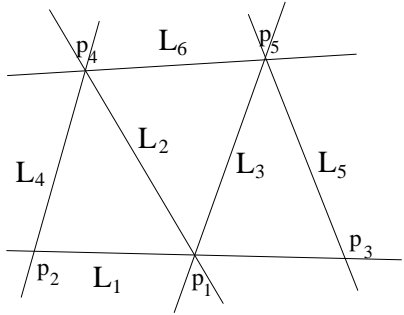}} 
\caption{Sextic curve for $(r, a, \delta)=(19, 3, 1)$} 
\label{r19}
\end{figure}


\section{Moduli of Borcea-Voisin threefolds}\label{BV 3-fold}

The unirationality of ${\moduli}$ implies that of the moduli of Borcea-Voisin threefolds. 
Let $(X, \iota)$ be a 2-elementary $K3$ surface and $E$ be an elliptic curve. 
The involution $(\iota, -1_E)$ of $X\times E$ extends to an involution $j$ of 
the blow-up $ \widetilde{X\times E}$ of  $X\times E$ along the fixed curve of $(\iota, -1_E)$. 
The quotient $Z = \widetilde{X\times E}/\langle j \rangle$ is a smooth Calabi-Yau threefold (\cite{Vo}, \cite{Borcea}). 
The projection $\widetilde{X\times E} \to X$ (resp. $\widetilde{X\times E} \to E$) induces a fibration 
$\pi_1\colon Z\to Y=X/\langle \iota \rangle$ (resp. $\pi_2\colon Z\to E/\langle -1_{E} \rangle$) 
with constant $E$-fiber (resp. $X$-fiber), 
whose discriminant locus is the branch locus of the quotient morphism 
$X \to Y$ (resp. $E \to E/\langle -1_E \rangle$). 
Following \cite{Yo2}, we call the triplet $(Z, \pi_1, \pi_2)$ 
the \textit{Borcea-Voisin threefold} associated to $(X, \iota)$ and $E$. 
Two Borcea-Voisin threefolds are isomorphic 
if and only if the corresponding 2-elementary $K3$ surfaces and 
elliptic curves are respectively isomorphic (\cite{Yo2}). 
The data $(\pi_1, \pi_2)$ may be regarded as a kind of polarization of $Z$, 
as the following remark shows.

\begin{lemma}\label{isom BV}
Let $(Z, \pi_1, \pi_2)$, $(Z', \pi_1', \pi_2')$ be Borcea-Voisin threefolds, 
and let $\Lambda$ (resp. $\Lambda'$) be the primitive closure of $\pi_1^{\ast}{\rm Pic}Y$ in ${\rm Pic}Z$ 
(resp. $(\pi_1')^{\ast}{\rm Pic}Y'$ in ${\rm Pic}Z'$). 
Then we have $(Z, \pi_1, \pi_2)\simeq (Z', \pi_1', \pi_2')$ if and only if 
we have $(Z, \Lambda) \simeq (Z', \Lambda')$. 
\end{lemma}

\begin{proof}
It suffices to prove the ``if'' part. 
Let $f\colon Z\to Z'$ be an isomorphism with $f^{\ast}\Lambda'=\Lambda$. 
There exist a very ample line bundle $H$ on $Y$ and a line bundle $H'$ on $Y'$ with 
$\pi_1^{\ast}H \simeq f^{\ast}(\pi_1')^{\ast}H'$. 
Since $|H| \simeq |\pi_1^{\ast}H| \simeq |(\pi_1')^{\ast}H'| \simeq |H'|$, 
we see that $H'$ is base point free. 
Via the projective morphisms $Z\to|\pi_1^{\ast}H|^{\vee}$ and $Z' \to|(\pi_1')^{\ast}H'|^{\vee}$, 
we obtain a morphism $g\colon Y'\to Y$ with $g \circ \pi_1' = \pi_1 \circ f^{-1}$. 
One checks that $g$ is bijective and hence is isomorphic. 
Considering the fibers and the discriminant loci of $\pi_1$ and $\pi_1'$, 
we obtain $E\simeq E'$ and $(X, \iota)\simeq(X', \iota')$. 
\end{proof}

The main invariant of a Borcea-Voisin threefold is defined as that of the associated 2-elementary $K3$ surface. 
Obviously, two Borcea-Voisin threefolds are deformation equivalent if and only if they have the same main invariant. 
Let $X(1)={\rm SL}_{2}({\Z}) \backslash {\mathbb H}$ be the moduli space of elliptic curves.  

\begin{theorem}[\cite{Yo2}]\label{moduli BV}
The variety ${\moduli}\times X(1)$ is a coarse moduli space of Borcea-Voisin threefolds of main invariant $(r, a, \delta)$. 
\end{theorem}

By Theorem \ref{main} we have the following. 

\begin{theorem}\label{unirat moduli BV}
The moduli spaces of Borcea-Voisin threefolds are unirational. 
\end{theorem}

\noindent
\textbf{Acknowledgement.}
I would like to express my gratitude to my advisor, Professor Ken-Ichi Yoshikawa, 
for suggesting this problem and for his help and encouragement. 
I am further indebted to him for generously writing his work in the appendix. 
I am grateful to Dr. S. Okawa for helpful discussion on invariant theory. 
I also thank Professors S. Kond\=o, Y. Miyaoka, S. Mukai, and T. Terasoma for their comments. 
Finally, I would like to thank a referee for helpful suggestions for improving and simplifying the exposition. 
This work was supported by Grant-in-Aid for JSPS fellows [21-978].

%
%


\appendix
\begin{Large}
\section{}\label{appendix}
\end{Large}

\vskip 5pt

\begin{center}
by  \ \ \  \begin{large}Ken-Ichi  Yoshikawa\end{large}\footnote{Research partially supported by the Grants-in-Aid for 
Scientific Research (B) 19340016, JSPS}  
\end{center}

\vskip 12pt

\par
In this note, we give a proof of the following result using automorphic forms. 

\begin{theorem}\label{theorem:Kodairadim}
The moduli space ${\mathcal M}_{(r,a,\delta)}$ has Kodaira dimension
$-\infty$ if either $13\leq r\leq17$ or $r+a=22$, $r\leq17$.
\end{theorem}

This is a consequence of the following criterion due to Gritsenko \cite{Gr} 
(the idea first appeared in \cite{GHS08}).

\begin{theorem}[Gritsenko]\label{theorem:Gritsenko} 
Let $L$ be a lattice of signature $(2,n)$ with $n\geq3$ and 
$\Gamma\subset{\rm O}(L)^+$ be a subgroup of finite index. 
Following \cite{GHS07}, let $R\subset\Omega_L^+$ denote the ramification divisor of the projection
$\pi\colon\Omega_L^+\to{\mathcal F}_L(\Gamma)$. 
Suppose we have an integer $\nu\geq0$ and an automorphic form $F_k$ 
on $\Omega_L^+$ for $\Gamma$ of weight $k$ such that $k\geq\nu n$ 
and that $\nu R-{\rm div}(F_k)$ is an effective divisor.  
If $k>\nu n$ or $\nu R-{\rm div}(F_k)\not=0$, then 
$$
\kappa({\mathcal F}_L(\Gamma))=-\infty.
$$
\end{theorem}

\begin{proof}
When $\nu=1$, the result is exactly \cite[Th.\,1.5]{Gr}. 
When $\nu>1$, the same proof works after replacing
$F_{nm}/F_k^m$ by $F_{nm}^{\nu}/F_{k}^{m}$ in the proof of
\cite[Th.\,1.5]{Gr}.
For the convenience of the reader, we give some detail. 
Assume $\omega\in H^{0}({\mathcal F}_{L}(\Gamma),mK_{{\mathcal F}_{L}(\Gamma)})$, $m>0$.
Regard $\Omega_{L}^{+}$ as a tube domain of ${\C}^{n}$. 
Then $\pi^{*}\omega=F_{nm}(z)\,(dz_{1}\wedge\ldots\wedge dz_{n})^{\otimes m}$,
where $F_{nm}(z)$ is a non-zero automorphic form on $\Omega_L^+$ for $\Gamma$ of weight $mn$. 
Since $\omega$ is holomorphic on ${\mathcal F}_L(\Gamma)$, 
$F_{nm}$ must vanish on $R$ at least of order $m$ (cf. \cite{GHS07}). 
Hence ${\rm div}(F_{nm})-mR\geq0$. 
Then $F_{nm}^{\nu}/F_{k}^{m}$ is an automorphic form for $\Gamma$ of weight 
$-m(k-\nu n)\leq0$ with effective divisor
$$
{\rm div}(F_{nm}^{\nu}/F_{k}^{m})\geq m(\nu R-{\rm div}(F_{k}))\geq0.
$$
Since $n\geq3$, $F_{nm}^{\nu}/F_k^m$ must be a constant. Hence
$k=\nu n$ and $\nu R={\rm div}(F_k)$, which contradicts the assumption.
\end{proof}

As an application of his criterion, 
Gritsenko gives several examples of orthogonal modular varieties with 
Kodaira dimension $-\infty$. See \cite{Gr} for those examples.
We thank Professor V.A. Gritsenko, whose lecture in the conference 
``Moduli and Discrete Groups'' at RIMS, Kyoto (2009) inspired this note and 
who kindly showed his paper \cite{Gr} when we wrote this note.

\subsection{The case $13\leq r\leq17$}

\begin{theorem}
\label{theorem:case1}
If $13\leq r\leq 17$, then $\kappa({\mathcal M}_{(r,a,\delta)})=-\infty$.
\end{theorem}

\begin{proof}
Let $L_-$ be the anti-invariant lattice of a $2$-elementary $K3$ surface of type $(r,a,\delta)$ with $r\geq11$. 
We denote $g=11-\frac{1}{2}(r+a)$. 
By \cite[Th.\,8.1]{Yo3}, there exists an automorphic form $\Psi_{L_-}$ for ${\rm O}(L_-)^+$ of weight
$k=(r-6)(2^g+1)$ with divisor
${\rm div}(\Psi_{L_{-}})=D'_{L_{-}}+(2^g+1)D''_{L_{-}}$,
where $D'_{L_{-}}$ and $D''_{L_{-}}$ are reduced divisors 
$$
D'_{L_{-}}
:=
\sum_{\lambda\in L_{-},\,\lambda^{2}=-2,\,\lambda/2\not\in L_{-}^{\lor}}
\lambda^{\perp},
\qquad
D''_{L_{-}}
:=
\sum_{\lambda\in L_{-},\,\lambda^{2}=-2,\,\lambda/2\in L_{-}^{\lor}}
\lambda^{\perp}.
$$
By definition, $D:=D'_{L_-}+D''_{L_-}$ is the discriminant divisor of $\Omega_{L_-}^+$.
Let $R\subset\Omega_{L_-}^+$ be the ramification divisor of the projection
$\Omega_{L_-}^+\to\mathcal{F}({\rm O}(L_-)^+)$. 
We set $\nu=2^g+1$ in Theorem~\ref{theorem:Gritsenko}. 
Since $n=20-r$ and $r\geq13$, we get
$k-\nu n=2\nu(r-13)\geq0$.
Since $R\geq D$ by \cite[Proof of Th.\,1.1.]{GHS07}, we get
$\nu\,R-{\rm div}(\Psi_{L_{-}}) \geq (\nu-1)D'_{L_{-}}\geq0$. 
When $r>13$ or $D'_{L_{-}}\not=0$, the result follows from 
Theorem~\ref{theorem:Gritsenko}. 
When $r=13$ and $D'_{L_-}=0$, then $L_-=U(2)\oplus M_7$. 
Let $r\in L_-$ be a vector with $r^2=-4$. 
Since the reflection with respect to $r$ is an element of ${\rm O}(L_-)^+$, 
we get $r^{\perp}\subset R$ and $r^{\perp}\not\subset D$, 
which implies $\nu\,R-{\rm div}(\Psi_{L_{-}})\not=0$. 
The result follows again from Theorem~\ref{theorem:Gritsenko}.
\end{proof}

\subsection{The case $r+a=22$ and $r\leq17$}
\par
We construct an automorphic form for ${\rm O}(L_-)^+$ 
satisfying the conditions in Theorem~\ref{theorem:Gritsenko} as a Borcherds product \cite{Borcherds98}. 
For this, we first construct a modular form of type $\rho_{L_-}$ 
with those properties required in \cite[Th.\,13.3]{Borcherds98}.
In what follows, we write
$r_-=r(L_-)$, $a_-=a(L_-)$, $\sigma_-=4-r_-$.
Let ${\rm Mp}_2({\Z})$ be the metaplectic double cover of ${\rm SL}_2({\Z})$,
which is generated by 
$S:=(\binom{0\,-1}{1\,\,\,\,0},\sqrt{\tau})$ and
$T:=(\binom{1\,1}{0\,1},1)$. 
See \cite[Sect.\,2]{Borcherds98} for more about ${\rm Mp}_2({\Z})$.
\par

\subsubsection
{Elliptic modular forms}
\par
We set $q=e^{2\pi i\tau}$ for $\tau\in{\mathbb H}$ and
$$
\eta(\tau)=q^{1/24}\prod_{n=1}^{\infty}(1-q^{n}),
\qquad
\theta_{\langle2\rangle}(\tau)=\sum_{n\in{\Z}}q^{n^2}, 
\qquad
\theta_{\langle2\rangle+1/2}(\tau)=\sum_{n\in{\Z}}q^{(n+\frac{1}{2})^2}.
$$
Set
${\rm M}\Gamma_{0}(4):=
\{(\binom{a\,b}{c\,d},\sqrt{c\tau+d})\in {\rm Mp}_2({\Z});\,c\equiv0\mod 4\}$.
By \cite[Lemma 5.2]{Borcherds00}, there exists
a character 
$\chi_{\theta}\colon {\rm M}\Gamma_{0}(4)\to\{\pm1,\pm i\}$ such that 
$\theta_{\langle2\rangle}(\tau)$ is a modular form for ${\rm M}\Gamma_{0}(4)$ 
of weight $1/2$ with character $\chi_{\theta}$.
\par
Set $\eta_{1^{-8}2^{8}4^{-8}}(\tau):=
\eta(\tau)^{-8}\eta(2\tau)^{8}\eta(4\tau)^{-8}$
and define 
$\psi_{m}(\tau)$, $m\in{\Z}$, by
$$
\psi_{m}(\tau)
:=
\eta_{1^{-8}2^{8}4^{-8}}(\tau)^{2}\,\theta_{\langle2\rangle}(\tau)^{8+m}
-
2(m+16)\,\eta_{1^{-8}2^{8}4^{-8}}(\tau)\,\theta_{\langle2\rangle}(\tau)^{m}.
$$
Since $\eta_{1^{-8}2^{8}4^{-8}}(\tau)$ is a modular form for ${\rm M}\Gamma_{0}(4)$ 
of weight $-4$ with trivial character, $\psi_{m}(\tau)$ is a modular form for 
${\rm M}\Gamma_{0}(4)$ of weight $\frac{m-8}{2}$ with character $\chi_{\theta}^{m}$.
Since $\eta_{1^{-8}2^{8}4^{-8}}(\tau)=q^{-1}+8+36q+O(q^{2})$ and
$\theta_{\langle2\rangle}(\tau)=1+2q+O(q^{4})$, we get 
$$
\psi_{m}(\tau)=q^{-2}+2(-m^{2}-9m+124)+O(q).
$$
Write $\psi_{m}(\tau)=\sum_{l\in{\Z}}d_{m}(l)\,q^{l}$ and define 
$h_{m}^{(i)}(\tau)$, $i\in{\Z}/4{\Z}$ 
as the series
$$
h_{m}^{(i)}(\tau):=
\sum_{l\equiv i\,\,{\rm mod}\,\,4}d_{m}(l)\,q^{l/4}.
$$
Then we have $\sum_{i\in{\Z}/4{\Z}}h^{(i)}_{m}(\tau)=\psi_{m}(\tau/4)$.

\subsubsection
{Vector-valued elliptic modular forms}
\par
Let ${\C}[D_{L_{-}}]$ be the group ring of the discriminant group $D_{L_{-}}$  
with the standard basis $\{{\bf e}_{\gamma}\}_{\gamma\in D_{L_{-}}}$. 
The Weil representation 
$\rho_{L_{-}}\colon {\rm Mp}_2({\Z})\to {\rm GL}({\C}[D_{L_{-}}])$ 
is defined as follows (cf. \cite[Sect.\,2]{Borcherds98}):
$$
\rho_{L_{-}}(T)\,{\bf e}_{\gamma}:=e^{\pi i\gamma^{2}}{\bf e}_{\gamma},
\qquad
\rho_{L_{-}}(S)\,{\bf e}_{\gamma}:=
\frac{i^{-\sigma_{-}/2}}{|D_{L_{-}}|^{1/2}}\sum_{\delta\in D_{L_{-}}}
e^{-2\pi i\langle\gamma,\delta\rangle}{\bf e}_{\delta}.
$$
We use the notion of modular forms of type $\rho_{L_{-}}$,
for which we refer to \cite[Sect.\,2]{Borcherds98}.
\par
Our construction is based on the following observation due to Borcherds.

\begin{proposition}
\label{proposition:Yoshikawa}
If $\phi(\tau)$ is a modular form for ${\rm M}\Gamma_{0}(4)$
with character $\chi_{\theta}^{\sigma_{-}}$, then
$$
{\mathcal B}_{L_{-}}[\phi](\tau)
:=
\sum_{g\in {\rm M}\Gamma_0(4)\backslash {\rm Mp}_2({\Z})}
\phi|_{g}(\tau)\,\rho_{L_{-}}(g^{-1})\,{\bf e}_{0}
$$ 
is a modular form for ${\rm Mp}_2({\Z})$ of type $\rho_{L_{-}}$ of 
the same weight as that of $\phi(\tau)$,
where $\phi|_{g}(\tau):=\phi(\frac{a\tau+b}{c\tau+d})\,(c\tau+d)^{-2l}$ for 
$g=\binom{a\,b}{c\,d}\in {\rm Mp}_2({\Z})$.
\end{proposition}

\begin{proof}
See e.g. \cite[Prop.\,7.1]{Yo3}.
\end{proof}

\par
Set $V:=S^{-1}T^{2}S=(\binom{\,\,1\,\,\,\,0}{-2\,\,1},\sqrt{-2\tau+1})$.
The coset ${\rm M}\Gamma_0(4)\backslash {\rm Mp}_2({\Z})$ 
is represented by $\{1,S,ST,ST^2,ST^3,V\}$.  
We define 
${\bf v}_{k}:=
\sum_{\delta\in D_{L_{-}},\,\delta^{2}\equiv k/2\,\,{\rm mod}\,\,2}
{\bf e}_{\delta}\in{\C}[D_{L_{-}}]$ for $k\in{\Z}/4{\Z}$.
Let ${\bf 1}_{L_{-}}\in D_{L_{-}}$ be the unique element such that
$\langle{\bf 1}_{L_{-}},\gamma\rangle=\gamma^{2}\mod{\Z}$
for all $\gamma\in D_{L_{-}}$.
By \cite[Proof of Lemma 7.5]{Yo3}, we get the following relations
$$
\rho_{L_{-}}((ST^{l})^{-1})\,{\bf e}_{0}
=
i^{\frac{\sigma_{-}}{2}}2^{-\frac{a_{-}}{2}}
\sum_{k=0}^{3}i^{-lk}\,{\bf v}_{k},
\qquad
\rho_{L_{-}}(V^{-1})\,{\bf e}_{0}
=
{\bf e}_{{\bf 1}_{L_{-}}},
$$
$$
\eta_{1^{-8}2^{8}4^{-8}}|_{ST^{l}}(\tau)
=
2^{4}\eta_{1^{-8}2^{8}4^{-8}}\left(\frac{\tau+l}{4}\right),
\qquad
\eta_{1^{-8}2^{8}4^{-8}}|_{V}(\tau)=-16\eta(2\tau)^{-16}\eta(4\tau)^{8},
$$ 
$$
\theta_{\langle2\rangle}|_{ST^{l}}(\tau)
=
(2i)^{-\frac{1}{2}}\theta_{\langle2\rangle}\left(\frac{\tau+l}{4}\right),
\qquad
\theta_{\langle2\rangle}|_{V}(\tau)
=
\theta_{\langle2\rangle+1/2}(\tau).
$$
Then we get
$$
\psi_{m}|_{ST^l}(\tau) = 2^{\frac{8-m}{2}}i^{-\frac{m}{2}}\,\psi_{m}\left(\frac{\tau+l}{4}\right).
$$
Since $\eta(2\tau)^{-16}\eta(4\tau)^{8}=1+O(q)$ and
$\theta_{\langle2\rangle+1/2}(\tau)=2q^{1/4}+O(q^{5/4})$, we get
$$
\psi_{m}|_{V}(\tau)=O(q^{m/4}).
$$
\par
In what follows, we assume $r_{-}<12$ and $m=8+\sigma_-$. 
Then 
$$
\begin{aligned}
\sum_{l=0}^{3}\psi_{m}|_{ST^{l}}(\tau)\,
\rho_{L_{-}}\left((ST^{l})^{-1}\right)\,{\bf e}_{0}
&=
2^{-\frac{\sigma_{-}+a_{-}}{2}}
\sum_{j=0}^{3}\sum_{l=0}^{3}\sum_{s\in{\Z}/4{\Z}}
h_{m}^{(s)}(\tau+l)\,i^{-lj}\,{\bf v}_{j}
\\
&=
2^{\frac{r_{-}-a_{-}}{2}}\sum_{j=0}^{3}h_{m}^{(j)}(\tau)\,{\bf v}_{j}.
\end{aligned}
$$
By Proposition~\ref{proposition:Yoshikawa}, 
${\mathcal B}_{L_{-}}[\psi_{8+\sigma_{-}}]$ is a modular form of type
$\rho_{L_{-}}$ of weight $\sigma_{-}/2$.
By the definition of ${\mathcal B}_{L_{-}}[\psi_{8+\sigma_{-}}]$
and the expansion of $h^{(l)}_{m}(\tau)$, we get the expansion
$$
\begin{aligned}
{\mathcal B}_{L_{-}}[\psi_{8+\sigma_{-}}](\tau)
&=
\psi_{8+\sigma_{-}}(\tau)\,{\bf e}_{0}
+
2^{\frac{r_{-}-a_{-}}{2}}\,
\sum_{l=0}^{3}h^{(l)}_{8+\sigma_{-}}(\tau)\,{\bf v}_{l}
+
\psi_{8+\sigma_{-}}|_{V}(\tau)\,{\bf e}_{{\bf 1}_{L_{-}}}
\\
&=
\left\{q^{-2}+2(-m^{2}-9m+124)+O(q)\right\}\,{\bf e}_{0}
\\
&\quad
+2^{\frac{r_{-}-a_{-}}{2}}\left\{2(-m^{2}-9m+124)+O(q)\right\}\,{\bf v}_{0}
+O(q^{1/4})\,{\bf v}_{1}
\\
&\quad
+2^{\frac{r_{-}-a_{-}}{2}}\{q^{-1/2}+O(q^{1/2})\}\,{\bf v}_{2}
+O(q^{3/4})\,{\bf v}_{3}
+O(q^{m/4})\,{\bf e}_{{\bf 1}_{L_{-}}}.
\end{aligned}
$$
\par
From the first equality, we see that ${\rm O}(L_-)$ preserves
${\mathcal B}_{L_{-}}[\psi_{8+\sigma_{-}}]$ (cf. \cite[Th.\,7.7 (2)]{Yo3}).
By \cite[Th.\,13.3]{Borcherds98}, the Borcherds lift
$\Xi_{L_{-}}:=\Psi_{L_{-}}(\cdot,{\mathcal B}_{L_{-}}[\psi_{8+\sigma_{-}}])$ 
is a holomorphic automorphic form on $\Omega_{L_{-}}^{+}$ 
for ${\rm O}(L_-)^+$ of weight $(2^{\frac{r_{-}-a_{-}}{2}}+1)(-m^{2}-9m+124)$
with zero divisor
$$
{\rm div}(\Xi_{L_{-}})
=
\sum_{\lambda\in L_{-},\,\lambda^{2}=-4}\lambda^{\perp}
+
2^{\frac{r_{-}-a_{-}}{2}}\,
\sum_{\lambda\in L_{-}^{\lor},\,\lambda^{2}=-1}\lambda^{\perp}.
$$

\begin{theorem}
\label{theorem:case2}
If $r+a=22$ and $11\leq r\leq17$, then $\kappa({\mathcal M}_{(r,a,\delta)})=-\infty$.
\end{theorem}

\begin{proof}
By the conditions $r+a=22$ and $11\leq r\leq17$, we get $r_-=a_-$ and $5\leq r_-\leq11$. 
We have an explicit expression
$L_{-}=\langle2\rangle^{2}\oplus\langle-2\rangle^{r_{-}-2}$, from which
we get $L_{-}^{\lor}=\frac{1}{2}L_{-}$. 
We set
${\mathcal H}:=\sum_{\lambda\in L_{-},\,\lambda^{2}=-4}\lambda^{\perp}$.
Then
${\rm div}(\Xi_{L_{-}})=2{\mathcal H}$. 
If $\lambda\in L_{-}$ and $\lambda^{2}=-4$, then the reflection
with respect to $\lambda$ is an element of ${\rm O}(L_-)^+$. 
Hence we get the inclusion of divisors $R\supset{\mathcal H}$,
which implies $R-{\mathcal H}\geq0$.
\par
We set $\nu=1$, $k=-m^{2}-9m+124$ and $F_{k}=\Xi_{L_{-}}^{1/2}$ in 
Theorem~\ref{theorem:Gritsenko}.
Since $n=r_{-}-2$, we get $k-n=-m^{2}-8m+114>0$ when $r_{-}\geq5$, i.e, 
$m\leq7$. Since ${\rm div}(F_{k})={\mathcal H}$, 
we get $R-{\rm div}(F_{k})\geq0$.
Now the result follows from Theorem~\ref{theorem:Gritsenko}.
\end{proof}


\end{document}